\documentclass{amsart}

\usepackage{tikz,wrapfig}
\usepackage{graphicx}
\usepackage{hyperref}
\usepackage{amsmath }
\usepackage{bbm,dsfont}
\usepackage{mathrsfs}
\usepackage[all]{xy}
\usepackage{amsfonts}
\usepackage{ifthen}
\usepackage{rotating}
\usepackage{amssymb}
\usepackage{enumerate}

\DeclareMathOperator{\rank}{rank}

\DeclareMathOperator{\is}{O}

\DeclareMathOperator{\im}{\rm im}

\newcommand{\C}{\mathbb{C}}
\newcommand{\Z}{\mathbb{Z}}
\newcommand{\lra}{\longrightarrow}
\newcommand{\ra}{\rightarrow}
\newcommand{\PP}{\mathbb{P}}
\newcommand{\R}{\mathbb{R}}

\newcommand{\Ps}{\mathbb P^}

\newcommand{\bbP}{\mathbb{P}}
\newcommand{\bbQ}{\mathbb{Q}}
\newcommand{\bbC}{\mathbb{C}}
\newcommand{\bbZ}{\mathbb{Z}}

\def\Ker{{\text{Ker}}}
\def\Hom{\mathop{\mathrm{Hom}}\nolimits}
\def\rk{\mathop{\mathrm{rank}}\nolimits}

\newtheorem{thm}{Theorem}[section]
\newtheorem{lemma}[thm]{Lemma}
\newtheorem{pro}[thm]{Proposition}
\newtheorem{cor}[thm]{Corollary}

\addtocounter{section}{-1}

\theoremstyle{definition}
\newtheorem{rem}[thm]{Remark}
\newtheorem{rem-def}[thm]{Remark-Definition}
\newtheorem{ex}[thm]{Example}

\begin{document}

\title{K3 surfaces with non-symplectic automorphisms of prime order}

\author{Michela Artebani, Alessandra Sarti, Shingo Taki \\
\ \\
\tiny{(with an Appendix by Shigeyuki Kond{$\bar{\rm o}$})}}

\begin{abstract}
In this paper we present the classification of non-symplectic automorphisms of prime order on  $K3$ surfaces, i.e.\,\,we describe the topological structure of their fixed locus and determine their invariant lattice in cohomology. 
We provide new results for automorphisms of order $5$ and $7$ and alternative proofs for higher orders.
Moreover, for any prime $p$, we identify the irreducible components of the moduli space of $K3$ surfaces with a non-symplectic automorphism of order $p$.
\end{abstract}

\thanks{The first author has been partially 
supported by Proyecto FONDECYT Regular 2009, N. 1090069.}

\maketitle

\pagestyle{myheadings} 
\markboth{M. Artebani, A. Sarti, S. Taki}{Non-symplectic automorphisms on $K3$ surfaces}
\setcounter{tocdepth}{1}

\section{Introduction}
A \textit{$K3$ surface} is a compact surface $X$ over $\mathbb{C}$ with trivial 
 canonical bundle and $\dim H^{1}(X, \mathcal{O}_{X})=0$.
 In the following, we will denote by $S_{X}$, $T_{X}$ and $\omega _{X}$ 
the Picard lattice, the transcendental lattice and a nowhere vanishing 
holomorphic $2$-form on $X$, respectively.

An automorphism $\sigma$ of a $K3$ surface $X$ is called \textit{symplectic} if it acts trivially on $\mathbb{C}\omega _{X}$.
This paper deals with $K3$ surfaces carrying  a \textit{non-symplectic} automorphism of prime order $p$.  
In this case it is known that $\sigma$ acts without non zero fixed vectors on the transcendental lattice, so that $T_X$ acquires the structure of a module over $\Z[\zeta_p]$, where $\zeta_p=e^{2\pi i/p}$. Since $T_X$ is contained in the rank $22$ lattice $H^2(X,\Z)$, this implies that $p$ is at most $19$.
 
Non-symplectic automorphisms have been studied by several authors, e.g. in \cite{V, N3, N4, AS, Taki, OZ1, OZ2,  OZ3,  MO}. 
It is known that for $p=13,17,19$ there are only isolated  pairs $(X_p,\sigma_p)$, where  $\sigma_p$ is a non-symplectic automorphism of order $p$ acting on the $K3$ surface $X_p$. This was first announced (without proof) by Vorontsov in \cite{V} and then it was proved by Kond\=o in \cite{Kondo1} and Oguiso-Zhang in \cite{OZ2}.  

In this paper we give the classification for $p=5,7$, we survey the known results for $p=2,3$ and we provide different proofs and examples for $p=11, 13, 17, 19$. The key idea is to characterize the fixed locus of the automorphism in terms of the properties of its invariant lattice in $H^2(X,\Z)$.
The main result is the following.
\begin{thm}\label{introthm} Let $S$ be a hyperbolic $p$-elementary lattice ($p$ prime) of rank $r$ with $\det(S)=p^a$. 
Then $S$ is isometric to the invariant lattice of a non-symplectic automorphism $\sigma$ of order $p$ on a $K3$ surface if and only if 
$$(*)\qquad 22-r-(p-1)a\in 2(p-1)\mathbb{Z}_{\geq 0}.  $$
%$m-a$ is  non-negative and even.

Moreover, if $\sigma$ is such automorphism, then its fixed locus  $X^{\sigma}$ is the disjoint union of smooth curves and isolated points and has the following form:
\begin{equation*}
 X^{\sigma}=
\begin{cases}
\emptyset &   \text{\rm if $S\cong U(2)\oplus E_{8}(2)$}, \\
E_{1}\cup E_{2} & \hspace{0cm} \text{\rm if $S\cong U\oplus E_{8}(2)$}, \\
C\cup R_{1}\cup\cdots\cup R_{k}\cup \{ p_{1}, \dots, p_{n} \} & \text{ otherwise},
\end{cases}
\end{equation*}
where  $E_i$ is a smooth elliptic curve,  $R_i$ is a smooth rational curve,  $p_{i}$ is an isolated fixed point and $C$ is a smooth curve of genus 
$$g=\dfrac{22-r-(p-1)a}{2(p-1)}.$$ 
Moreover:

\[\begin{array}{c|cccccc}
 & p=2 & p=3,5,7 & p=11 & p=13 & p=17 & p=19 \\[6pt]
 \hline \\[-7pt]
 
n & 0 & \dfrac{-2+(p-2)r}{p-1}  & \dfrac{2+9r}{10} & 9 & 7 & 5\\ [13pt]
k &  \dfrac{r-a}{2}& \dfrac{2+r-(p-1)a}{2(p-1)} &\dfrac{-2+r-10a}{20}  & 1 & 0 & 0\\[8pt]
\end{array}\]

with the convention that $X^{\sigma}$ contains no fixed curves if $k=-1$.
\end{thm}
As a consequence of this result, we determine the maximal components of the moduli space of $K3$ surfaces with a non-symplectic automorphism of order $p$ for any prime $p$. 
Moreover, we show that for $p>2$ the topological structure of the fixed locus of $\sigma$ determines uniquely the action of $\sigma$ on $H^2(X,\Z)$ (see Remark \ref{obs}).\\

The plan of the paper is the following. 
Section 1 introduces some background material on lattices. In Section 2 we prove the main properties of the invariant lattice $S(\sigma)$ and of the fixed locus of a non-symplectic automorphism $\sigma$ of prime order.
Theorem \ref{point} gives the number of isolated fixed points and the local action at them, as a function of the rank of $S(\sigma)$.
Moreover, by means of Smith exact sequences, we determine the maximal genus $g$ of a fixed curve as a function of the lattice invariants of $S(\sigma)$, see Corollary \ref{genus}.
These results show that the topology of the fixed locus of the automorphism is uniquely determined by the isometry class of its invariant lattice and allow to give a complete classification.
The following sections give more explicit results for each prime order $p$.

Sections 3 and 4 briefly review the known results for $p=2$ and $3$ respectively.  All possible configurations of the fixed locus related to the invariants of $S(\sigma)$ are represented in Figures \ref{ord2} and \ref{ord3}.

In section 5 we classify non-symplectic automorphisms of order $5$. The classification theorem is resumed in Table \ref{ord5}.
The topological structure of the fixed locus gives a natural stratification in $7$ families. Two of them, of dimensions $3$ and $4$, are the maximal irreducible components of the moduli space (as we will show in section 10). We provide projective models for the generic member of each family.

In section 6 we give a similar classification and description for $p=7$, see Table \ref{ord7}. In this case there are two maximal components of dimension $2$.

In sections 7 and 8 we provide an alternative view of the classification for $p=11, 13, 17, 19$.
In case $p=11$ it is known that the moduli space has two maximal 1-dimensional components, while the pair $(X,\sigma)$ is unique for $p>11$.

In section 9 we deal with moduli spaces. First, we recall the structure of the moduli space of pairs $(X,\sigma)$ where $X$ is a $K3$ surface and $\sigma$ is a non-symplectic automorphism of order $p$ with a given representation in $H^2(X,\Z)$. This moduli space is known to be isomorphic to a quotient of an open dense subset of either an Hermitian symmetric domain of type $IV$ (for $p=2$) or of a complex ball (for $p>2$) for the action of a discrete group.
For any prime $p$ we identify the irreducible components of the moduli space of $K3$ surfaces with a non-symplectic automorphism of order $p$, see Theorem \ref{moduli1}.

In the appendix S. Kond\=o shows that the moduli space of pairs $(X,\sigma)$, where $\sigma$ is non-symplectic of order $7$ having only isolated fixed points, is a ball quotient isomorphic to the Naruki $K3$ surface. A similar example for $p=5$ was given by the same author in \cite{Kondo2}.\\

{\it Acknowledgements.} We would like to thank Igor Dolgachev, Alice Garbagnati and the referee for several helpful comments.

\section{Lattices}
A lattice is a finitely generated free abelian group equipped with a non-degenerate symmetric bilinear form with integer values. If the signature of the lattice is $(1, r-1)$ then it is called hyperbolic.  We will work with even lattices i.e. such that the quadratic form on it takes values in $2\Z$.

The quadratic form on $L$ determines a canonical embedding $L\subset L^{\ast }=\Hom (L,\mathbb{Z})$. We denote by $A_{L}$ the factor group $L^{\ast }/L$, which is a finite abelian group.
If this group is trivial, then $L$ is called \textit{unimodular}.

Let $p$ be a prime number. A lattice $L$ is called \textit{$p$-elementary} if $A_{L}\simeq \mathbb{Z}_p^{a}$. 
If $L$ is a $p$-elementary lattice primitively embedded in a unimodular lattice $M$ and $L^{\perp}$ is its orthogonal complement in $M$, then it is known that $L^{\perp}$ is also $p$-elementary and  $p^a=|\det(L)|=|\det(L^{\perp})|$. 

The following result classifies even, indefinite, $p$-elementary lattices (see \cite{RS}).
\begin{thm}\label{RS}
An even, indefinite, $p$-elementary lattice of rank $r$ for $p \neq 2$ and $r\geq 2$ is 
uniquely determined by the integer $a$.

For $p\neq 2$ a hyperbolic $p$-elementary lattice with invariants $a,r$  exists if and only if
the following conditions are satisfied:
$a\leq r$, $r\equiv 0 \pmod 2$ and 

\begin{equation*}
\begin{cases}
\text{for } a\equiv 0 \pmod 2, & r\equiv 2 \pmod 4 \\
\text{for } a\equiv 1 \pmod 2, & p\equiv (-1)^{r/2-1} \pmod 4 .
\end{cases}
\end{equation*}
Moreover $r>a>0$, if $r\not \equiv 2 \pmod 8$.

An even,  indefinite, $2$-elementary lattice is determined by $r, a$ and a third invariant $\delta\in \{0,1\}$, see \cite{N3}.
 \end{thm}
 
\subsection*{Notation}
 We will denote by $U$ the unique even unimodular hyperbolic lattice of rank two and by $A_{m}$, $D_{n}$, $E_{l}$ the even, 
negative definite lattices associated with the Dynkin diagrams of the corresponding types ($m\geq 1$,  $n\geq 4$, $l=6,7,8$).
Moreover, $L(a)$ and $L^b$ will denote the lattices whose bilinear form is respectively the one on $L$ multiplied by $a$ and the orthogonal sum of $b$ copies of the one on $L$. 
 \subsection*{Examples}
 \begin{enumerate}[$-$]
 \item The lattices $U$ and $E_8$ are unimodular. Any even unimodular lattice of signature $(3,19)$ is isometric to $L_{K3}=U^{\oplus 3}\oplus E_{8}^{\oplus 2}$ (\cite{Milnor,  Serre}).
\item If $p$ is prime, then the lattice  $A_{p-1}$ is $p$-elementary with $a=1$.   
\item The lattice $E_7$ is $2$-elementary with $a=1$.
 \item
If $p\equiv 3\, (\!\!\!\mod\, 4)$, then the lattice
$$K_{p}=\left(\begin{array}{cc}
-(p+1)/2 & 1\\
1& -2
\end{array}
\right)
$$
is negative definite, $p$-elementary, with $a=1$. Note that $K_{3}\cong A_2$.
\item If $p\equiv 1\, (\!\!\!\mod\, 4)$ then the lattice
$$H_{p}=\left(\begin{array}{cc}
(p-1)/2 &1\\
1& -2
\end{array}
\right)
$$
is hyperbolic, $p$-elementary, with $a=1$.  
\item The lattice  
$$
A_4^*(5)=\left(\begin{array}{rrrr}
-4 &1&1&1\\
1& -4&1&1\\
1&1&-4&1\\
1&1&1&-4
\end{array}
\right)
$$ 
is negative definite,  $5$-elementary with $a=3$.
\item The lattice
$$
L_{17}=\left(\begin{array}{rrrr}
-2&1&0&1\\
1&-2&0&0\\
0&0&-2&1\\
1&0&1&-4
\end{array}
\right)
$$ 
is negative definite,  $17$-elementary with $a=1$.
 \end{enumerate}

\section{Non-symplectic automorphisms on $K3$'s}
Let $X$ be a $K3$ surface i.e. a simply connected smooth compact complex surface with a nowhere vanishing holomorphic 2-form $\omega_X$. 
The cohomology group $H^2(X,\Z)$, equipped with the cup product, is known to be an unimodular lattice isometric to $L_{K3}$.
The Picard lattice $S_X$ and the transcendental lattice $T_X$ are the following primitive sublattices of $H^2(X,\Z)$:
$$S_X=\{x\in H^2(X,\Z): (x,\omega_X)=0\},\hspace{1cm}  T_X=S^{\perp}_X .$$ 
 
An automorphism $\sigma$ of $X$ is called \emph{non-symplectic} if its action on the vector space $H^{2,0}(X)=\C\omega_X$ is not trivial. Observe that, by \cite[Theorem 3.1]{N1}, $K3$ surfaces with a non-symplectic automorphism of finite order are always algebraic. 
In this paper we are interested in non-symplectic automorphisms of prime order i.e.
$$ \sigma^p=id \hspace{0.5cm}\mbox{ and }\hspace{0.5cm} \sigma^*(\omega_X)=\zeta^k_p \omega_X,\ 0<k<p,$$
where $\zeta_p$ is a primitive $p$-th root of unity.

The automorphism $\sigma$ induces an isometry $\sigma^*$ on $H^2(X,\Z)$ which preserves  both $S_X$ and $T_X$.
We will consider the invariant lattice and its orthogonal in $H^2(X,\Z)$:
$$S(\sigma)=\{x\in H^2(X,\Z): \sigma^*(x)=x\},\hspace{1cm}  T(\sigma)=S(\sigma)^{\perp}.$$

\begin{thm}\label{ns}
Let $X$ be a $K3$ surface and $\sigma$ be a non-symplectic automorphism of $X$ of prime order $p$.
Then
\begin{enumerate}[a)]
\item $S(\sigma)\subset S_X$ and $T_X\subset T(\sigma)$;
\item $T(\sigma)$ and $T_X$ are free modules over $\Z[\zeta_p]$ via the action of $\sigma^*$;
\item $S(\sigma)$ and $T(\sigma)$ are $p$-elementary lattices and $A_{S(\sigma)}\cong A_{T(\sigma)}\cong \mathbb{Z}_p^{a}$ with $$a\leq \frac{ \rk (T(\sigma))}{p-1}.$$ 
\end{enumerate}
\end{thm}

\proof The statements $a), b)$ and the first claim in $c)$ are proved in \cite[Section 3]{N1} or \cite[Lemma 1.1]{MO}. For the inequality in $c)$, we will generalize the proof of \cite[Claim 3.4]{MO} as follows.
By point $b)$ we have that $T(\sigma)\cong \Z[\zeta_p]^m$ as a $\Z[\zeta_p]$-module.
Let $e_1,\dots, e_m$ be a basis of $T(\sigma)$ over  $\Z[\zeta_p]$ and let
$$\{b_{ij}:\ i=1,\dots,m,\ j=0,\dots,p-2\}$$ be the corresponding $\Z$ basis of $T(\sigma)$. Since $T(\sigma)$ is $p$-elementary, then any $y\in T(\sigma)^*$ is of the form
$$y=\frac{1}{p}\sum_{i,j} y_{ij}b_{ij},\quad y_{ij}\in \Z.$$
Moreover, since $\sigma^*=id$ on $S(\sigma)$, then $\sigma^*=id$ on $A_{S(\sigma)}\cong A_{T(\sigma)}$. Thus  modulo $T(\sigma)$ we have
$$0\equiv \sigma^*(y)-y= \frac{1}{p}\sum_{i=1}^m(\sum_{j=0}^{p-3} y_{ij}b_{i\, j+1}-y_{i \,p-2}(b_{i 0}+\cdots+b_{i\, p-2})-\sum_{j=0}^{p-2} y_{ij}b_{ij}  ).$$
From the vanishing of the coefficients of the $b_{ij}$'s  it follows that 
$y_{ij}\equiv (j+1)y_{i0} $ modulo $p$.
Thus
$$\frac{1}{p}\sum_{j=0}^{p-2} y_{ij}b_{ij}\equiv y_{i0}[\frac{1}{p}\sum_{j=0}^{p-2}(j+1)b_{ij}]=y_{i0}B_i.$$
 This implies that $A_{T(\sigma)}$ is generated by  $B_1,\dots,B_m$, hence $a\leq m$.\qed\\

\noindent In what follows we will denote by $m=(22-r)/(p-1)$ the rank of $T(\sigma)$ as a $\Z[\zeta_p]$-module ($r$ denotes the rank of $S(\sigma)$).\\

We will now describe the structure of the fixed locus $X^{\sigma}$ of a non-symplectic automorphism $\sigma$ of order $p$ of a $K3$ surface. We can assume $\sigma$ to act on $\omega_X$ as the multiplication by $\zeta_p$.
The action of $\sigma$ can be locally linearized and diagonalized at a fixed point $x\in X^{\sigma}$ (see \S 5, \cite{N1}), so that its possible local actions are

$$ A_{p,t}=\left(\begin{array}{cc}
	\zeta_p^{t+1} &0\\
0&\zeta_p^{p-t}
\end{array}
\right),\ \ t=0,\dots,p-2.$$
If $t=0$ then $x$ belongs to a smooth fixed curve for $\sigma$, otherwise $x$ is an isolated fixed point. We will say that an isolated point $x\in X^{\sigma}$ is of \emph{type }$t$ ($t>0$) if the local action at $x$ is given by $A_{p,t}$ and we will denote by $n_{t}$ the number of isolated points of $\sigma$ of type $t$. 
\begin{lemma}\label{fix}
The fixed locus of $\sigma$ is either empty or the disjoint union of isolated  points and smooth curves. Moreover, in the second case, $X^{\sigma}$ is either the union of two disjoint elliptic curves or of the form
\begin{equation}\label{fixl}X^{\sigma}=C\cup R_1\cup\cdots\cup R_{k}\cup\{p_1,\dots,p_n\},\end{equation}
where $C$ is a smooth curve of genus $g\geq 0$, $R_i$ is a smooth rational curve and $p_i$ is an isolated point.
\end{lemma}
\proof
The first statement follows from the previous discussion about the local action of $\sigma$.
By Hodge index theorem the Picard lattice of $X$ is hyperbolic. Thus, if $X^{\sigma}$ contains a smooth curve $C$ of genus $g>1$, then the other curves in the fixed locus are rational (by adjunction formula, since their class have negative self-intersection).

If $X^{\sigma}$ contains an elliptic curve, then the other fixed curves can be either rational or elliptic.  
Assume that there are two fixed elliptic curves $E_1,E_2$. Since $E_1,E_2$ are disjoint, then their classes are linearly equivalent and define a $\sigma$-invariant elliptic fibration $\varphi:X\rightarrow \PP^1$. Since the local action of $\sigma$ at $p\in E_1$ is of type $(x,y)\mapsto (x,\zeta_p y)$, then $\sigma$ induces a non-trivial action on $\PP^1$.
Thus $\sigma$ has exactly two fixed points in $\PP^1$ and its fixed locus is equal to $E_1\cup E_2$. This concludes the proof.
\qed\\

The aim of the rest of this section is to relate the topological invariants $g,k,n$ of the fixed locus $X^{\sigma}$ to the lattice invariants $m, a$ (and so $r,a$) of $T(\sigma)$.
The methods we will apply generalize techniques in  \cite{Kh} and \cite{N3}. 

\begin{lemma}\label{chi} The Euler characteristic of $X^{\sigma}$ is $24-mp.$
\end{lemma}
\proof
By the topological Lefschetz formula we have:
\[ \chi (X^{\sigma}) =\sum _{i=0}^{4}(-1)^{i} \text{tr}(\sigma ^{\ast }|H^{i}(X, \mathbb{R}))
 = 2+\text{tr}(\sigma ^{\ast }|S(\sigma))+\text{tr}(\sigma^{\ast }|T(\sigma))=2+r-m,\]
 where $r$ is the rank of $S(\sigma)$. This gives the statement since $m(p-1)=22-r$.
 \qed\\

\noindent   Observe that, by Lemma \ref{fix}, the Euler characteristic of $X^{\sigma}$ is either zero or
$\mathcal X(X^{\sigma})=(2-2g)+2k+n$. In the second case, let $\alpha=1-g+k.$
\begin{thm}\label{point}
Let $\sigma$ be a non-symplectic automorphism of prime order $p$ of a $K3$ surface and let $r$ be the rank of its invariant lattice $S(\sigma)$. 
Then the types $n_i$ of the isolated fixed points of $\sigma$ and the integer $\alpha$ can be expressed in  function of $r$ as in Table \ref{t}.
 \end{thm}
\begin{center} 
\begin{table}[h]
 \begin{tabular}{|c || c || c | c | c | c | c | c | c | c | c  | c |}
\hline 
$\scriptstyle{p}$ & $\scriptstyle{\alpha}$ & $\scriptstyle{n_{1}}$ & $\scriptstyle{n_{2}}$ &$\scriptstyle{n_{3}}$ &$\scriptstyle{n_{4}}$ &$\scriptstyle{n_{5}}$ &$\scriptstyle{n_{6}}$ &$\scriptstyle{n_{7}}$ &$\scriptstyle{n_{8}}$ & $\scriptstyle{n_{9}}$ & $\scriptstyle{n}$   \\ [1ex]
\hline
$\scriptstyle{2}$ & $\scriptstyle{r-10}$ & & & & & & & & & & $\scriptstyle{0}$ \\ [1ex]
$\scriptstyle{3}$ & $\frac{r-8}{2}$ &  $\scriptstyle{\alpha+3}$ & & & & & & & & & $\scriptstyle{\alpha+3}$   \\ [1ex]
$\scriptstyle{5}$ & $\frac{r-6}{4} $ &   $\scriptstyle{2\alpha+3}$ & $\scriptstyle{ 1+\alpha}$   && &&&&&  &   $\scriptstyle{3\alpha+4}$    \\ [1ex]
$\scriptstyle{7}$ & $ \frac{r-4}{6} $ & $\scriptstyle{2\alpha+2}$ & $\scriptstyle{ 1+2\alpha}$ & $\scriptstyle{ \alpha}$ &&&&&& & $\scriptstyle{5\alpha+3}$ \\ [1ex]
$\scriptstyle{11}$ & $\frac{r-2}{10} $   & $\scriptstyle{1+2\alpha}$ & $\scriptstyle{2\alpha}$ & $\scriptstyle{ 2\alpha}$ & $\scriptstyle{ 1+2\alpha}$& $\scriptstyle{\alpha} $   & & & &   & $\scriptstyle{9\alpha+2}$ \\[1ex]
$\scriptstyle{13}$ & $\frac{r+2}{12}$ & $\scriptstyle{1+2\alpha}$ & $\scriptstyle{1+2\alpha}$ & $\scriptstyle{ 2\alpha}$ & $\scriptstyle{ 2\alpha-1}$ & $\scriptstyle{2\alpha-2} $& $\scriptstyle{\alpha-1}$ & & & & $\scriptstyle{11\alpha-2}$  \\[1ex]
$\scriptstyle{17}$ & $\frac{r-6}{16}$ & $\scriptstyle{2\alpha}$& $\scriptstyle{2\alpha}$& $\scriptstyle{2\alpha}$& $\scriptstyle{2\alpha}$& $\scriptstyle{2\alpha+1}$ & $\scriptstyle{2\alpha+2}$ & $\scriptstyle{ 2\alpha+3}$ & $\scriptstyle{ \alpha+1}$ &  & $\scriptstyle{15\alpha+7}$  \\ [1ex]
$\scriptstyle{19}$ & $\frac{r-4}{18}$ & $\scriptstyle{2\alpha}$ & $\scriptstyle{ 2\alpha} $ & $ \scriptstyle{2\alpha} $ & $ \scriptstyle{2\alpha+1} $ & $ \scriptstyle{2\alpha+2} $ & $ \scriptstyle{2\alpha+1} $ & $ \scriptstyle{2\alpha+1} $ & $\scriptstyle{ 2\alpha} $ & $\scriptstyle{ \alpha}$ & $\scriptstyle{17\alpha+5}$  \\ [1ex]
\hline
\end{tabular}
\vspace{0.5cm}
\caption{Isolated fixed points}\label{t}
\end{table}
\end{center}
\vspace{-1cm}

\proof 
The holomorphic Lefschetz formula \cite[Theorem 4.6]{AS2} allows to compute the holomorphic Lefschetz number $L(\sigma )$ of $\sigma$ in two ways. First we have that
\[ L(\sigma ) = \sum _{i=0}^{2} (-1)^i\text{tr}(\sigma ^{\ast }|H^{i}(X, \mathcal{O}_{X})). \]
By Serre duality $H^{2}(X, \mathcal{O}_{X})\simeq H^{0}(X,\mathcal{O}_{X}(K_{X}))^{\vee }$, so that  
\begin{equation}\label{l1}
L(\sigma )=1+\zeta_p ^{p-1}.
\end{equation}
On the other hand, if the fixed locus is as in (\ref{fixl}), we also have that
\[ L(\sigma ) = \sum_{t=1}^{p-2} n_{t} a(t)+b(g)+kb(0), \]
where  
\begin{equation}\label{l2}
 a(t)=\frac{1}{\det(I-\sigma^{\ast }|T_{t} )}
  =\frac{1}{\det(I-A_{p,t})} 
  =\frac{1}{(1-\zeta ^{t})(1-\zeta^{p-t+1})}, 
 \end{equation}
 with $T_{t}$ the tangent space of $X$ at a point of type $t$, and
\begin{equation}\label{l3}
b(g) =\frac{1-g}{1-\zeta_p}- \frac{\zeta_p (2g-2)}{(1-\zeta_p)^{2}} \\
  =\frac{(1+\zeta_p)(1-g)}{(1-\zeta_p )^{2}},\quad  b(0)=\frac{1+\zeta_p}{(1-\zeta_p)^2}.
\end{equation}
 If $X^{\sigma}$ is either empty or the union of two elliptic curves, then $L(\sigma)=0$.

Combining (\ref{l1}), (\ref{l2}) and (\ref{l3}) we get  the types $n_i$ appearing in Table \ref{t}.
 Moreover, since $\chi (X^{\sigma})=2\alpha+n$, we get the values of $\alpha$ in Table \ref{t} by applying  Lemma \ref{chi}. \qed\\
 
 We now assume that the fixed locus is as in (\ref{fixl}) and we will  determine the genus $g$ as a function of the invariants $m,a$ of $T(\sigma)$ by means of Smith exact sequences.
 We will consider the following isometries of $H^2(X,\Z)$:
 \[
 \frak{g} =1+\sigma^* +{(\sigma^*)}^{2}+\cdots+{(\sigma^*)}^{p-1} 
 ,\quad \frak{h} =1-\sigma^*.
\]
Observe that $\ker \frak{h} =S(\sigma)$, $\ker \frak{g} =T(\sigma)$ and $|H^{2}(X,\mathbb{Z})/S(\sigma)\oplus T(\sigma)|=p^{a}$. We now consider  the coefficient homomorphism 
\[
c:H^{2}(X, \mathbb{Z})\longrightarrow H^{2}(X, \mathbb{Z}_p)
.\]
Observe that  $c(S(\sigma)\oplus T(\sigma))$
 coincides with $E=\ker \mathfrak g\subset H^2(X,\Z_p)$. This implies that $a=\dim H^{2}(X, \mathbb{Z}_p)-\dim E$.

Let $C(X)$ be the chain complex of $X$ with coefficients in $\mathbb{Z}_p$. 
The automorphism $\sigma$ acts on $C(X)$ and gives rise to chain subcomplexes $\frak{g}  C(X)$ and $\frak{h} C(X)$. 
We denote by $H_{i}^{\frak{g} }(X)$, $H_{i}^{\frak{h} }(X)$ the associated Smith special homology groups with coefficients in $\mathbb{Z}_p$ as in \cite[Definition 3.2, Ch.~III]{Bredon} and by $\chi ^{\frak{g} }(X)$, $\chi ^{\frak{h} }(X)$ the corresponding Euler characteristics. 
By \cite[(3.4), Ch.~III]{Bredon} there is an isomorphism
$$H_{i}^{\frak{g} }(X)\simeq H_{i}(X/\langle \sigma \rangle, X^{\sigma }),$$
where $X^{\sigma }$ is identified with its image in the quotient surface $X/\langle \sigma \rangle$.

In what follows, the coefficients are intended to be in $\Z_p$. Observe that $\frak{g} =\frak{h}^{p-1} $ over $\mathbb{Z}_p$.
Let $\rho=\frak{h} ^{i}$ and $\bar \rho=\frak{h} ^{p-i}$,
then for any $i,j=1,\dots, p-1$ we have the exact triangles (\cite[Theorem 3.3 and (3.8), Ch.~III]{Bredon})
$$\xymatrix{
(\text{T1})\ \ & &H_*(X) \ar[dl]_{\rho_*} &   \\
&H_*^{\rho}(X)\ar[rr]_{}& &H_*^{\bar\rho}(X)\oplus H_*(X^{\sigma })\ar[ul]_{i_*}\\  
(\text{T2})\ \ & &H_*^{\frak{h} ^{j}}(X) \ar[dl]_{\frak{h}_*} &   \\
&H_*^{\frak{h} ^{j+1}}(X)\ar[rr]_{}& &H_*^{\frak{g} }(X) \ar[ul]_{i_*}} 
$$
where $\frak{h}_*$, $i_*$ and $\rho_*$ have degree $0$ and the horizontal arrows have degree $-1$.
The triangle (T1) induces two long homology sequences for $\rho=\frak{g}$ and $\frak{h}$. 
In particular, since $H_1(X)=H_3(X)=0$, then (T1) induces the  sequences:
 \[
0\rightarrow H_{3}^{\frak{g} }(X)
\stackrel{\gamma_{3}}{\rightarrow }H_{2}^{\frak{h} }(X)\oplus H_{2}(X^{\sigma })
\stackrel{\alpha _{2}}{\rightarrow } H_{2}(X)
\stackrel{\beta _{2}}{\rightarrow } H_{2}^{\frak{g} }(X)
\stackrel{\gamma_{2}}{\rightarrow }H_{1}^{\frak{h} }(X)\oplus H_{1}(X^{\sigma })\rightarrow 0,
\]
\[
0\rightarrow H_{3}^{\frak{h} }(X)
\stackrel{\gamma_{3}'}{\rightarrow }H_{2}^{\frak{g} }(X)\oplus H_{2}(X^{\sigma })
\stackrel{\alpha _{2}'}{\rightarrow } H_{2}(X)
\stackrel{\beta _{2}'}{\rightarrow } H_{2}^{\frak{h} }(X)
\stackrel{\gamma_{2}'}{\rightarrow }H_{1}^{\frak{g} }(X)\oplus H_{1}(X^{\sigma })\rightarrow 0.
\]

\begin{lemma}\label{kye1}
$H_{0}^{\frak{g}}(X)=H_{0}^{\frak{h}}(X)=0$, $\dim H_i^\frak{h}(X)=\dim H_i^\frak{g}(X)=1$ for $i=1,3,4$.
\end{lemma}
\begin{proof}
Since $X/\langle \sigma \rangle$ is connected, then
$H_{0}^{\frak{g}}(X)\simeq H_{0}(X/\langle \sigma \rangle, X^{\sigma })=0$ (see \cite[Prop. 13.10]{Gr}). This implies that $\dim H_{0}^{\frak{h}}(X)=0$ by repeated use of (T2). By (T1), this gives $\dim H_1^\frak{h}(X)=\dim H_1^\frak{g}(X)$.
 
Moroever, by (T1) we get $\dim H_{3}^{\frak{h} }(X)=\dim H_{3}^{\frak{g} }(X)$ and $\dim H_4^\frak{g}(X)+\dim H_4^\frak{h}(X)-\dim H_3^\frak{g}(X)=1.$
The exact sequence of the pair $(X/\langle \sigma \rangle, X^{\sigma})$ gives $\dim H_4^\frak{g}(X)=\dim H_4(X/\langle \sigma\rangle)=1$, which implies that
$\dim H_4^\frak{h}(X)=\dim H_3^\frak{g}(X)=\dim H_3^\frak{h}(X)$.
Observe that $\dim H_4^\frak{h}(X)\leq 1$ by (T1) (has an injective map to $H^4(X)$) and $\dim H_4^\frak{h}(X)\geq 1$ by (T2) (there is an injective map from $H_4^\frak{g}(X)$ to $H_4^\frak{h}(X)$).
This gives the statement.
\end{proof}
\begin{lemma}\label{kye3}
$i_*:H_2^{\frak g}\ra H_2^{\frak h}$ is injective.
\end{lemma}
\begin{proof}
Assume that $i_*$ is not injective. By \cite[(3.7), Ch.~III]{Bredon} we have the following exact square:
\begin{equation}\label{square}
\xymatrix{0\ar[r]& H_{3}^{\frak{g} }(X)\ar[r]^{\gamma_{3}\quad \quad} & H_{2}^{\frak{h} }(X)\oplus H_{2}(X^{\sigma })\\
0\ar[r]& H_{3}^{\frak{h} }(X)
\ar[r]^{\gamma_{3}'\qquad} \ar[u]^{\frak h_*^{p-2}}&H_{2}^{\frak{g} }(X)\oplus H_{2}(X^{\sigma }).\ar[u]_{i_*\oplus 1}}
\end{equation}
Thus $\frak h_*^{p-2}$ is the zero homomorphism, since $\dim H_3^\frak{h}(X)=\dim H_3^\frak{g}(X)=1$ by Lemma \ref{kye1}.  
It can be easily seen that this leads to a contradiction looking at the first terms in the homology ladders associated to the diagrams \cite[(3.6) and (3.7), Ch.~III]{Bredon}.
\end{proof}

\begin{lemma}\label{kye2}
$\dim \im \alpha _{2}=\dim E$, $\dim \im \alpha _{2}- \dim \im \alpha _{2}'=\dim H_{2}^{\frak{h} }(X)-\dim H_{2}^{\frak{g} }(X)$.
\end{lemma}
\begin{proof}
We will denote by $E'= \ker \frak{g}_*\subset H_2(X)$. Observe that $E'$ is the dual of $E$, so that $\dim E=\dim  E'$.

We will use the following properties of Smith exact sequences: $\im\alpha_2\subset E'$  and the projection of $\gamma_{3}$ on the second factor is the boundary homomorphism in the sequence of the relative homology $H_{i}(X/\langle \sigma \rangle, X^{\sigma })$. 
The latter property implies that such projection is injective, since $H_{3}(X/\langle \sigma \rangle)=0$

If $x\in E'$ then $\alpha'_{2}(\beta_{2}(x)\oplus 0)=\frak{g}_*(x)=0$. 
Hence $\beta_2(x)\oplus 0\in \ker \alpha'_2=\im \gamma'_{3}$. By the square (\ref{square}), we get that $i_*(\beta_2(x))\oplus 0=\gamma_3(y)\in \im(\gamma_3)$.
 Since the projection of $\gamma_3$ on the second factor is injective, then $\gamma_3(y)=i_*(\beta_2(x))=0$.
 By Lemma \ref{kye3}, this implies that $x\in \ker(\beta_2)=\im(\alpha_{2})$, proving that $E'\subset \im\alpha_2$.

The two exact sequences induced by (T1) and Lemma \ref{kye1} give the second statement.  
\end{proof}

\begin{pro}\label{smith}
$\sum_i\dim H_{i}(X^{\sigma },\Z_p)=24-2a-m(p-2). $
\end{pro}
\begin{proof}
By (T2) we have $\chi ^{\frak{h} ^{j}}(X) =\chi ^{\frak{g}}(X)+\chi ^{\frak{h} ^{j+1}}(X)$ for any $j=1,\dots,p-1$.
Since $\frak{h} ^{p-1} =\frak{g} $, then $\chi ^{\frak{h} }(X) =(p-1)\chi ^{\frak{g} }(X)$.
Moreover, by  (T1) we get:
\[\chi (X)-\chi (X^{\sigma })= \chi ^{\frak{g} }(X)+\chi ^{\frak{h} }(X) = p\chi ^{\frak{g} }(X). \]
Thus by  Lemma \ref{chi} we have
\[
\chi ^{\frak{g} }(X)=m.
\]
On the other hand 
$\chi ^{\frak{g} }(X)-\chi ^{\frak{h} }(X) =\dim H_{2}^{\frak{g} }(X)-\dim H_{2}^{\frak{h} }(X)$ by Lemma \ref{kye1}.
Hence, by Lemma \ref{kye2}, we get
\[
\dim \im \alpha _{2}- \dim \im \alpha _{2}'=\chi ^{\frak{h} }(X)-\chi ^{\frak{g} }(X)=\chi^\frak{g}(X)(p-2)=m(p-2).
\]

Then from the two exact sequences,  Lemma \ref{kye1} and  \ref{kye2} 
we have 
\begin{align*}
 &\sum_i\dim H_{i }(X)-\sum_i\dim H_{i }(X^{\sigma}) 
 =\chi ^{\frak{g} }(X)+\chi ^{\frak{h} }(X)-2\dim H_{1}(X^{\sigma })  \\
  &=\dim \im \beta _{2}+\dim \im \beta _{2}' 
  =2(\dim H_{2}(X)-\dim \im \alpha_2)+\dim \im \alpha _{2}-\dim \im \alpha _{2}'\\
 &=2a+m(p-2).
\end{align*}
This concludes the proof since $\sum_i \dim H_i(X)=24$.
 \end{proof}

\begin{cor}\label{genus} If the fixed locus of $\sigma$ is as in (\ref{fixl}), then $2g=m-a$.
Otherwise, if it is either empty or the union of two elliptic curves, then $m=a$ and $m-a=4$ respectively.  

\end{cor}
\begin{proof} It follows from Lemma \ref{chi} and Proposition \ref{smith}, since 
$$ \sum_i\dim H_i(X^{\sigma},\Z_p)-\chi(X^{\sigma},\Z_p)=2\dim H_1(X^{\sigma},\Z_p)=4g$$ 
in case the fixed locus is as in $(\ref{fixl})$ (and equals $0$ and $8$ if it is empty or the union of two elliptic curves respectively). \end{proof}

\begin{rem}\label{obs}
By Corollary \ref{genus} the integer $m-a$ is positive and even, this is equivalent to the condition $(*)$ in Theorem \ref{introthm}.

Moreover, by Lemma \ref{chi}, Theorem \ref{point} and Corollary \ref{genus} it follows that the invariants $g,k,n$ of the fixed locus of $\sigma$ uniquely determine the invariants $r,a$ (or $m,a$) of the invariant lattice and viceversa.
 Thus, by Theorem \ref{RS}, it is equivalent to give the topology of the fixed locus of $\sigma$ or its invariant lattice in $H^2(X,\Z)$ if $p\geq 3$.
An easy computation gives the formulas for $n,k$ as functions of $r,a$ given in Theorem  \ref{introthm}.

\end{rem}

\section{Order 2} We briefly recall the classification theorem for non-symplectic involutions on $K3$ surfaces  given by Nikulin in \cite[\S 4]{N3} and \cite[\S 4]{N4}.
The local action of a non-symplectic involution $\sigma$ at a fixed point is of type
$$A_{2,0}= \left(\begin{array}{cc}
1 &0\\
0& -1
\end{array}
\right),$$
so that $X^{\sigma}$ is the disjoint union of smooth curves and there are no isolated fixed points.
The lattice $S(\sigma)$ is $2$-elementary thus, according to Theorem \ref{RS}, its isometry class is determined by the invariants $r, a$ and $\delta$.

 \begin{figure}[h]
$$\begin{array}{cccccccccccccccccccccccr}
 &\ &\ &&&& &&\ &\ &&&&&&&&&&&&& \bullet\ \ \delta=1\\
&\ &\ &&&& &&\ &\ &&&&&&&&&&&&&\ast\ \ \delta=0
\end{array}$$
\vspace{-1.1cm}

\begin{tikzpicture}[scale=.43]
\filldraw [black] 
(1,1) circle (1.5pt)  node[below=-0.55cm]{10}
(2,0) node[below=-0.20cm]{*} 
(2,2) circle (1.5pt) node[below=-0.5cm]{9} node[below=-0.15cm]{*}
 (3,1) circle (1.5pt)
 (3,3) circle (1.5pt)node[below=-0.5cm]{8}
 (4,2) circle (1.5pt)
(4,4) circle (1.5pt)node[below=-0.5cm]{7}
(5,3) circle (1.5pt)
(5,5) circle (1.5pt)node[below=-0.5cm]{6}
(6,4) circle (1.5pt)node[below=-0.15cm]{*}
(6,2)    node[below=-0.23cm]{*}
(6,6) circle (1.5pt)node[below=-0.5cm]{5}
(7,3) circle (1.5pt)
(7,5) circle (1.5pt)
(7,7) circle (1.5pt)node[below=-0.5cm]{4}
(8,2) circle (1.5pt)
(8,4) circle (1.5pt)
(8,6) circle (1.5pt)
(8,8) circle (1.5pt)node[below=-0.5cm]{3}
(9,1) circle (1.5pt)
(9,3) circle (1.5pt)
(9,5) circle (1.5pt)
(9,7) circle (1.5pt)
(9,9) circle (1.5pt)node[below=-0.5cm]{2}
(10,0)  node[below=-0.20cm]{*}
(10,2) circle (1.5pt)node[below=-0.15cm]{*}
(10,4) circle (1.5pt)node[below=-0.15cm]{*}
(10,6) circle (1.5pt)node[below=-0.15cm]{*}
(10,8) circle (1.5pt)node[below=-0.15cm]{*}
(10,10) circle (1.5pt) node[below=-0.5cm]{1}node[below=-0.15cm]{*}
(11,1) circle (1.5pt)
(11,3) circle (1.5pt)
(11,5) circle (1.5pt)
(11,7) circle (1.5pt)
(11,9) circle (1.5pt)
(11,11) circle (1.5pt) node[below=-0.5cm]{0}
(12,2) circle (1.5pt)
(12,4) circle (1.5pt)
(12,6) circle (1.5pt)
(12,8) circle (1.5pt)
(12,10) circle (1.5pt) node[below=-0.5cm]{1}
(13,3) circle (1.5pt)
(13,5) circle (1.5pt)
(13,7) circle (1.5pt)
(13,9) circle (1.5pt) node[below=-0.5cm]{2}
(14,2)  node[below=-0.15cm]{*}
(14,4) circle (1.5pt)node[below=-0.15cm]{*}
(14,6) circle (1.5pt)node[below=-0.15cm]{*}
(14,8) circle (1.5pt) node[below=-0.5cm]{3}
(15,3) circle (1.5pt)
(15,5) circle (1.5pt)
(15,7) circle (1.5pt) node[below=-0.5cm]{4}
(16,2) circle (1.5pt)
(16,4) circle (1.5pt)
(16,6) circle (1.5pt) node[below=-0.5cm]{5}
(17,1) circle (1.5pt)
(17,3) circle (1.5pt)
(17,5) circle (1.5pt)node[below=-0.5cm]{6}
(18,0)  node[below=-0.2cm]{*}
(18,2) circle (1.5pt)node[below=-0.15cm]{*}
(18,4) circle (1.5pt) node[below=-0.5cm]{7}node[below=-0.15cm]{*}
(19,1) circle (1.5pt)
(19,3) circle (1.5pt) node[below=-0.5cm]{8}
(20,2) circle (1.5pt) node[below=-0.5cm]{9}
 ; 
\draw plot[mark=*] file {data/ScatterPlotExampleData.data};
\draw[->] (0,0) -- coordinate (x axis mid) (22,0);
    \draw[->] (0,0) -- coordinate (y axis mid)(0,12);
    \foreach \x in {0,1,2,3,4,5,6,7,8,9,10,11,12,13,14,15,16,17,18,19,20}
        \draw [xshift=0cm](\x cm,0pt) -- (\x cm,-3pt)
         node[anchor=north] {$\x$};
          \foreach \y in {1,2,3,4,5,6,7,8,9,10,11}
        \draw (1pt,\y cm) -- (-3pt,\y cm) node[anchor=east] {$\y$};
    \node[below=0.2cm, right=4.5cm] at (x axis mid) {$r$};
    \node[left=0.5cm, below=-2.7cm, rotate=90] at (y axis mid) {$a$};
 \draw[<-, blue](0.1,0.1)-- node[below=2cm,left=2cm]{$g$} (11,11);   
 \draw[<-, red](21.5,0.5)-- node[below=2cm,right=2.1cm]{$k$} (11,11);
  \end{tikzpicture} 
\caption{Order 2}\label{ord2}
\end{figure}
\begin{thm}[Theorem 4.2.2, \cite{N3}] \label{o2}
The fixed locus of a non-symplectic involution on a $K3$ surface is
\begin{itemize}
\item empty if $r=10$, $a=10$ and $\delta=0$,
\item the disjoint union of two elliptic curves if  $r=10$, $a=8$ and $\delta=0$,
\item the disjoint union of a curve of genus $g$ and $k$ rational curves otherwise, where
$$g=(22-r-a)/2,\quad k=(r-a)/2.$$
\end{itemize}
\end{thm}
Figure \ref{ord2} shows all the values of the triple $(r,a,\delta)$ which are realized and the corresponding invariants $(g,k)$ of the fixed locus.

The surfaces which arise as quotients of $K3$ surfaces by non-symplectic involutions have been classified in \cite{N4, AN, Zhang}. These are Enriques surfaces if $X^{\sigma}=\emptyset$ and smooth rational surfaces otherwise.

\section{Order 3}
Non-symplectic automorphisms of order $3$ on K3 surfaces have been recently classified in \cite{AS} and \cite{Taki}.
In this case the local action at a fixed point is of one of the following  
 $$ A_{3,1}=\left(\begin{array}{cc}
\zeta_3^2&0\\
0&\zeta_3^2
\end{array}
\right),\quad A_{3,0}= \left(\begin{array}{cc}
\zeta_3&0\\
0&1
\end{array}
\right)
,$$
so that the fixed locus contains both smooth curves and isolated points. Let $(r,a)$ be the invariants of $S(\sigma)$, and  let $\rk T(\sigma)=2m$ as in section \S3.
\begin{thm}[Table 1,\,\cite{AS} and Theorem 1.2,\,\cite{Taki}] \label{o3}
The fixed locus of a non-symplectic automorphism of order $3$ on a $K3$ surface is not empty and it is either:
\begin{itemize}
\item the union of three isolated points if $m=a=7$, or 
\item the disjoint union of $n$ points, a smooth curve of genus $g$ and $k$ smooth rational curves, where 
$$n = 10-m,\quad g = (m-a)/2,\quad k = 6-(m+a)/2.$$
\end{itemize}
\end{thm} 
All values of the pair $(m,a)$ and the corresponding invariants $(g,k,n)$ of the fixed locus
are represented in Figure \ref{ord3}.

\begin{figure}[h]
\begin{tikzpicture}[scale=.5]
\filldraw [black] 
(1,1) circle (1.5pt) node[below=-0.5cm]{5}
(2,2) circle (1.5pt) node[below=-0.5cm]{4}
(2,0) circle (1.5pt) 
(3,1) circle (1.5pt)
(3,3) circle (1.5pt) node[below=-0.5cm]{3}
(4,4) circle (1.5pt) node[below=-0.5cm]{2}
(4,2) circle (1.5pt)
(5,5) circle (1.5pt) node[below=-0.5cm]{1}
(5,3) circle (1.5pt)
(5,1) circle (1.5pt)
(6,6) circle (1.5pt) node[left=0.15cm, below=-0.4cm]{0} node[right=0.24cm,below=-0.24cm]{}
(6,4) circle (1.5pt)
(6,2) circle (1.5pt)
(6,0) circle (1.5pt)
(7,7) circle (1.5pt)node[below=-0.5cm]{}
(7,5) circle (1.5pt)node[ below=-0.5cm]{1}
(7,3) circle (1.5pt)
(7,1) circle (1.5pt)
(8,4) circle (1.5pt)node[ below=-0.5cm]{2}
(8,2) circle (1.5pt)
(9,3) circle (1.5pt)node[ below=-0.5cm]{3}
(9,1) circle (1.5pt)
(10,2) circle (1.5pt)node[ below=-0.5cm]{4}
(10,0) circle (1.5pt) node[ below=-0.5cm, green]{0}
 (9,0) circle (0pt) node[ below=-0.5cm, green]{}
(8,0) circle (0pt) node[ below=-0.5cm, green]{2}
(7,0) circle (0pt) node[ below=-0.5cm, green]{}
(6,0) circle (0pt) node[ below=-0.5cm, green]{4}
(5,0) circle (0pt) node[ below=-0.5cm, green]{}
(4,0) circle (0pt) node[ below=-0.5cm, green]{6}
(3,0) circle (0pt) node[ below=-0.5cm, green]{}
(2,0) circle (0pt) node[ below=-0.5cm, green]{8}
(1,0) circle (0pt) node[ below=-0.5cm, green]{}
; 
\draw plot[mark=*] file {data/ScatterPlotExampleData.data};
\draw[->] (0,0) -- coordinate (x axis mid) (12,0);
    \draw[->] (0,0) -- coordinate (y axis mid)(0,8);
    \foreach \x in {0,1,2,3,4,5,6,7,8,9,10}
        \draw [xshift=0cm](\x cm,1pt) -- (\x cm,-3pt)
         node[anchor=north] {$\x$};
          \foreach \y in {1,2,3,4,5,6,7}
        \draw (1pt,\y cm) -- (-3pt,\y cm) node[anchor=east] {$\y$};
    \node[below=0.2cm, right=3cm] at (x axis mid) {$m$};
    \node[below=-0.2cm, right=3cm, green] at (x axis mid) {$n$};
    \node[left=0.5cm, below=-2cm, rotate=90] at (y axis mid) {$a$};
 \draw[<-, blue](0.5,0.5)-- node[below=1.4cm,left=1.5cm]{$k$} (7,7);   
 \draw[<-, red](11.5,0.5)-- node[below=1.2cm,right=1.3cm]{$g$} (6,6);
  \end{tikzpicture} 
\caption{Order 3}\label{ord3}
\end{figure}
\begin{rem}
We give here a construction relating $K3$ surfaces with $n=3$, $k=0$, $g=1$ to those with $n=3$ and no fixed curves (cf. \cite[Proposition 4.7]{AS}). 
  Consider the elliptic $K3$ surface $X_{a,b}$  defined by the Weierstrass equation
$$
y^2=x^3+(t^6+a_1t^3+a_2)x+(t^{12}+b_1t^9+b_2t^6+b_3t^3+b_4)
,$$
where $a\in \C^2$, $b\in \C^4$ are generic.
The fibration has $24$ fibers of Kodaira type $I_1$ over the zeros of its discriminant polynomial. The automorphism of order $3$
$$
\sigma(x,y,t)=(x,y,\zeta_3 t)
$$  
acts non trivially on the basis of the fibration and preserves the smooth elliptic curves over $t=0$ and $t=\infty$. Moreover, the fiber over $t=0$ is pointwise fixed. This implies that $\sigma$ is non-symplectic and by Theorem \ref{o3} we have $n=3$, $k=0$, $g=1$. 
The minimal resolution of the quotient surface   $X_{a,b}/\langle\sigma\rangle$ is a rational elliptic surface $\pi_{a,b}:Y_{a,b}\ra \PP^1$ with Weierstrass equation
$$
y^2=x^3+(t^2+a_1t+a_2)x+(t^{4}+b_1t^3+b_2t^2+b_3t+b_4)
.$$ 
This fibration has one singular fiber of Kodaira type $IV$ over $t=\infty$ and $8$ fibers of type $I_1$. 
Let $p:P\ra \PP^1$ be a non-trivial principal homogeneous space of $\pi_{a,b}$ given by an order $3$ element in the fiber $\pi_{a,b}^{-1}(0)$ (see \cite[Ch.V, \S 4]{CD}). Then surface $P$ is a rational elliptic surface with a multiple smooth elliptic fiber of multiplicity $3$ over $0$. Let $Z$ be the surface obtained by blowing up the intersection point of the three rational curves of the fiber $p^{-1}(\infty)$ and then blowing down the strict transforms of the three rational curves. Then $Z$ is a log Enriques surface of index $3$ with three singular points. Let $X$ be the canonical cover of $Z$ and $\bar\sigma$ be a generator of the covering transformation group of the cover. By \cite[Theorem 5.1]{Zhang2},  since $Z$ has three singular points of type $(3,1)$ (these are Hirzebruch-Jung singularities, cf. \cite[Section 5]{BPV}) then $X$ is a $K3$ surface and $\bar\sigma$ is an order 3 non-symplectic automorphism of $X$ with only three fixed points.

A similar construction relates Examples \ref{fibr1}, \ref{7fam1}, \ref{11fam1} to Examples \ref{log5}, \ref{7fam2}, \ref{11fam2} respectively.  
\end{rem}
\section{Order 5}
An order five non-symplectic automorphism of a $K3$ surface has three possible local actions at a fixed point
$$ A_{5,0}=\left(\begin{array}{cc}
\zeta_5 &0\\
0& 1
\end{array}
\right),\ \ A_{5,2}=\left(\begin{array}{cc}
\zeta_5^3 &0\\
0& \zeta_5^3
\end{array}
\right)
,\ \   A_{5,1}=\left(\begin{array}{cc}
\zeta_5^2 &0\\
0& \zeta_5^4
\end{array}
\right) .$$
Thus the fixed locus can contain both fixed curves and isolated points of two different types.
We start providing two families of examples.
 \begin{ex}\label{fam1}
Let $\mathcal A$ be the family of plane sextic curves defined by  
 $$x_0(x_0-x_1)\prod_{i=1}^4(x_0-\lambda_ix_1)+x_2^5x_1=0,$$
where $\lambda_i\in \C$.
Observe that the projective transformation 
\begin{equation}\label{aut1}\bar\sigma(x_{0}, x_{1}, x_{2})=(x_{0}, x_{1}, \zeta_5 x_{2})\end{equation} 
preserves any
curve in $\mathcal A$. If $C\in\mathcal A $ is smooth, then the double cover $X$ of $\mathbb{P}^{2}$ branched along  $C$ is a $K3$ surface and $\bar\sigma$ induces an automorphism $\sigma$ of order $5$ on $X$.
Since $X^{\bar\sigma}=\{(0,0,1) \}\cup \{x_{2}=0 \}$, then $X^{\sigma}$
is the union of an isolated fixed point  and a smooth curve of genus two. This implies that $\sigma$ is non-symplectic (since symplectic automorphisms only have isolated fixed points)\\
If the complex numbers $\lambda_i$'s are distinct, then the corresponding sextic $C\in \mathcal A$  is smooth and has six fixed points on the line $x_2=0$.
Otherwise, if two or three of the $\lambda_i$'s coincide, then $C$ has a singular point of type $A_4$ or $E_8$ respectively. \\
More in detail, we have the following cases: 
$$\begin{array}{lcc}
&\mbox{Equation of $C$:} & \mbox{Types of singular points:}\\[10pt]
 a) &  x_0^2(x_0-x_1)\prod_{i=2}^4(x_0-\lambda_ix_1)+x_2^5x_1=0 & A_4 \\[4pt]

b) &x_0^3(x_0-x_1) (x_0-\lambda_3x_1)(x_0-\lambda_4x_1)+x_2^5x_1=0 & E_8 \\[4pt]

c) &x_0^2(x_0-x_1)^2 (x_0-\lambda_3x_1)(x_0-\lambda_4x_1)+x_2^5x_1=0 & A_4^2 \\[4pt]

d) & x_0^3(x_0-x_1)^2(x_0-\lambda_4x_1)+x_2^5x_1=0 & A_4\oplus E_8\\ [4pt]

e) & x_0^3(x_0-x_1)^3+x_2^5x_1=0  & E_8^2.\\ [4pt]
\end{array}
$$
\end{ex}
 
\begin{ex}\label{fam2} 
Let $\mathcal B$ be the family of plane sextic curves defined by
$$a_1x_0^6+a_2x_0^3x_1x_2^2+a_3x_0^2x_1^3x_2+x_0(a_4x_1^5+a_5x_2^5)+a_6x_1^2x_2^4=0,$$
where $a\in \C^6$. 
The projective transformation
\begin{equation}\label{aut2}\bar \sigma(x_0,x_1,x_2)=(x_0,\zeta_5x_1,\zeta_5^2x_2)\end{equation}
preserves any curve in $\mathcal B$. If $C\in\mathcal B$ is smooth, then the double cover $X$ of $\mathbb{P}^{2}$ branched along  $C$ is a $K3$ surface and $\bar\sigma$ induces a non-symplectic automorphism $\sigma$ of order $5$ on $X$. In fact $\sigma$ does not leave invariant the holomorphic 2-form (written in local coordinates):
$$
\frac{dx\wedge dy}{\sqrt{f}},
$$
where $f$ denotes the equation of a curve in $\mathcal B$  in local coordinates $x,y$. Since $X^{\bar\sigma}=\{(1,0,0) \}\cup\{(0,1,0) \}\cup\{(0,0,1) \}$ and all but one of these points belong to $C$,  then $X^{\sigma}$ is the union of $4$ isolated points.
 \end{ex}

\begin{table}

\begin{tabular}{c|c|c|c|c|c}
$n_1$ & $n_2$ & $g$ & $k$&  $T(\sigma)$ & $S(\sigma)$\\
\hline
 1 & 0 & 2 &0 &$H_5\oplus U\oplus E_8\oplus E_8$& $H_5$\\ 
 \hline
 3& 1& 1& 0  &$H_5\oplus U\oplus E_8\oplus A_4$& $H_5\oplus A_4$\\ 
 3& 1& - & - &$H_5\oplus U(5)\oplus E_8\oplus A_4$& $H_5\oplus A_4^*(5)$\\
\hline 
 5& 2& 1 & 1  &$U\oplus H_5\oplus E_8$ & $H_5\oplus E_8$\\
 5& 2& 0&  0  &$U\oplus H_5\oplus A_4^2$ & $ H_5\oplus A_4^2$\\
\hline 
 7& 3& 0 & 1&  $U\oplus H_5\oplus A_4$& $H_5\oplus A_4\oplus E_8$\\ 
  \hline
 9& 4 &0 & 2   &$U\oplus H_5 $& $ H_5\oplus E_8\oplus E_8$\\ 
  \end{tabular}\\

 \ \\
\ \\
  \caption{Order 5}\label{ord5}

\end{table}

\begin{thm}\label{o5}  The fixed locus of  a non-symplectic automorphism $\sigma$ of order $5$  on a $K3$ surface is either
\begin{itemize}
\item the union of three isolated points of type $1$ and one point of type $2$ if $S(\sigma)$ is isomorphic to $H_5\oplus A_4^*(5)$, or
\item the disjoint union of $n_i$ isolated points of type $i$, a smooth curve of genus $g$ and $k$ smooth rational curves where $(g,k,n_i)$ appears in a row of Table \ref{ord5}.
\end{itemize}
 The same table gives the corresponding invariant lattice $S(\sigma)$ and its orthogonal complement $T(\sigma)$.
\end{thm}
\proof 
 The list of lattices in Table \ref{ord5} is obtained by using Theorem \ref{RS} and Corollary \ref{genus} (which implies that $m-a$ is a positive even integer).
For any such lattice $S(\sigma)$ the invariants $g,k, n_i$ of $X^{\sigma}$ can be computed by means of Lemma \ref{chi}, Theorem \ref{point} and Corollary \ref{genus}.
\qed\\

We will now show that  a $K3$ surface $X$ with a non-symplectic automorphism $\sigma$ of order $5$ such that $S_X=S(\sigma)$ belongs to one of the two families in Examples \ref{fam1}, \ref{fam2}.  This will also show that all cases in Table \ref{ord5} do appear.

By Theorem \ref{o5} the invariant lattice of $\sigma$ is one of the following types.\\
  
\noindent \fbox{ $S(\sigma)\cong H_5$}  Let  $h$ be the generator of $H_5$ with $h^2=2$. By \cite[Lemma 3.5]{MO}  we can assume $h$ to be ample and  base point free.  Thus the morphism associated to $|h|$  is a double cover of $\Ps2$ branched along a smooth plane sextic $B$. Moreover, since $h\in S(\sigma)$, then the action of $\sigma$ on $X$ induces a projectivity $\bar\sigma$ of $\Ps2$. By Theorem \ref{o5} the fixed locus of $\sigma$ contains a curve of genus $g=2$. This implies that $\bar\sigma$ has a curve in the fixed locus, hence for a suitable choice of coordinates it is of the form \eqref{aut1}.
 Thus $(X,\sigma)$ belongs to the family in Example \ref{fam1}.\\

\noindent \fbox{$S(\sigma)\cong H_5\oplus  A_4$} In this case we can assume  $h$  to be nef and  base point free, by a reasoning similar to the one in the proof of \cite[Lemma 3.5]{MO}. The associated morphism is a double cover of $\Ps2$ branched along a plane sextic $B$ with a singular point $p$ of type $A_4$. By  Theorem \ref{o5} the fixed locus of $\sigma$ contains a fixed curve. As before, this implies that $\bar\sigma$ is of type (\ref{aut1}). This implies that $X$ belongs to the family in  Example \ref{fam1}, a).  \\
   
\noindent The cases $H_5\oplus E_8$, $H_5\oplus A_4^2$, $H_5\oplus A_4\oplus E_8$ and $H_5\oplus E_8^2$ can be discussed in a similar way and correspond to Examples \ref{fam1} $c), d)$ and $e)$ respectively.\\

\noindent \fbox{$S(\sigma)\cong H_5\oplus A_4^*(5)$} Since $A_4^*(5)$ does not contain vectors with self-intersection $-2$ (see the description of $A_4^*$ in \cite{lat}), then we can assume $h$ to be ample as before and the associated morphism is a double cover branched along a smooth plane sextic $B$. 
By Theorem \ref{o5}, the automorphism $\sigma$ has at most isolated fixed points. Thus the same is true for $\bar \sigma$ so that, for a suitable choice of coordinates, it is of the form (\ref{aut2}).
Hence $(X,\sigma)$ is as in Example \ref{fam2}.

\begin{rem}
In \cite{OZ3} K. Oguiso and D-Q. Zhang showed the uniqueness of a $K3$ surface with a non-symplectic automorphism of order five and fixed locus containing no curves of genus $\geq 2$ and at least $3$ rational curves. In their approach they used log Enriques surfaces.
\end{rem}

\begin{rem}  
In \cite[\S 3.1, 3.2]{Kondo2} and in \cite[Remark 6]{K}  S. Kond\=o  considers the minimal resolution $X$ of the double cover of $\mathbb{P}^{2}$ branched along the union of the line $x_2=0$ and the plane quintic curve   
\[ x_{2}^{5}=\prod_{i=1}^{5}(x_{0}-\lambda _{i}x_{1}), \]
where the $\lambda _{i}$'s are distinct complex numbers. 
Then $X$ is a $K3$ surface with an automorphism $\sigma$ induced from $\bar\sigma$ as in (\ref{aut1}) with $n=7$, $g=0$, $k=0$ and $T(\sigma)\cong U\oplus H_5\oplus A_4^{2}$ as computed in \cite[Section 12, p. 53]{DK}. Since this family has dimension $2$, this gives a different model for the family of $K3$ surfaces in Example \ref{fam1}, $c)$.

In \cite[\S 7,(7.6)]{Kondo1}  appears the following elliptic $K3$ surface with order $5$ automorphism:
$$y^{2}=x^{3}+t^{3}x+t^{7},\quad \sigma (x, y, t)=(\zeta_5 ^{3}x, \zeta_5 ^{2}y, \zeta_5 ^{2}t).$$
Here the fixed locus has the invariants $n=13$, $g=0$ and $k=2$, hence this gives a different model for Example \ref{fam1}, $e)$.
\end{rem}

\begin{ex}\label{fibr1}
Let $X_{\alpha,\beta,\gamma}$ be an elliptic $K3$ surface with Weierstrass equation
$$
y^2=x^3+(t^5+\alpha)x+(\beta t^{10}+t^5+\gamma)
,\quad \alpha, \beta, \gamma\in \C.$$
For generic $\alpha,\beta$ and $\gamma$ the fibration has a fiber of Kodaira type $IV$ over $t=\infty$ and $20$ fibers of type $I_1$. Observe that the fibration has an automorphism of order $5$:
$$
\sigma(x,y,t)=(x,y,\zeta_5 t)
.$$ 
This automorphism acts non trivially on the basis of the fibration and preserves the fibers over $t=0$ and $t=\infty$. 
In fact it fixes pointwise the smooth fiber over $t=0$. This implies, by looking at the local action at a fixed point, that  $\sigma$ is non-symplectic. Checking in Table \ref{ord5} one can see that the fixed locus has $n_1=3$, $n_2=1$, $g=1$ and $k=0$. In fact it is easy to see that this is an equation of the generic $K3$ surface in the family of Example \ref{fam1}, a). Moreover observe that:
\begin{enumerate}[1)]
\item  If $\beta=0$ then the fibration has a singular fiber of Kodaira type $III^*$ over $t=\infty$ and $15$ fibers of type $I_1$. This corresponds to the case $n_1=5$, $n_2=2$, $g=1$, $k=1$ in Table \ref{ord5}.
\item If $\alpha^3=-27/4 \gamma^2$ then the fibration has a singular fiber of type $IV$ over $t=\infty$ and of type $I_5$ over $t=0$. This corresponds to $n_1=5$, $n_2=2$, $g=k=0$.
\item  If $\alpha=\gamma=0$ then there is a fiber of type $IV$ over $t=\infty$ and a fiber of type $II^*$ over $t=0$. This corresponds to $n_1=7$, $n_2=3$, $g=0$, $k=1$.
\item  If $\alpha=\beta=\gamma=0$ then there is a fiber of type $III^*$ over $t=\infty$ and a fiber of type $II^*$ over $t=0$. This corresponds to $n_1=9$, $n_2=4$, $g=0$, $k=2$.
\end{enumerate}
\end{ex}
 
 We recall that a \emph{log-Enriques surface} is a normal algebraic surface $Y$ having at most quotient singularities such that $h^1(Y,\mathcal O_Y)=0$ and $mK_Y=\mathcal O_Y$ for some positive integer $m$.
The \emph{index} of $Y$ is the smallest positive $m$ with this property.
The canonical covering  $q:\tilde Y\ra Y$ (i.e. that induced by the relation $mK_Y=\mathcal O_Y$) is a cyclic cover of degree $m$, \'etale over the smooth points of $Y$. The minimal resolution of $\tilde Y$ is known to be either a $K3$ surface or an abelian surface. 
\begin{rem} The general theory of log Enriques surfaces has been developed in \cite{Zhang1, Zhang2, Zhang3}. Moreover, log Enriques surfaces of index $2$ have been studied in \cite{Zhang}, of index $5$ in \cite{OZ3}, of index $11$ in \cite{OZ1}, and of index  $13,17,19$ in \cite{OZ2}.
 \end{rem}
\begin{ex}\label{log5}
Let $\pi_{a,b,c}:Y_{a,b,c}\ra \Ps 1$ be the rational jacobian elliptic surface with Weierstrass equation
$$y^2=x^3+tx+(at^2+bt+c),\ \ a,b,c\in \C.$$
For generic $a,b,c$ the elliptic fibration has a fiber of Kodaira type $IV^*$ at $t=\infty$ and $4$ fibers of type $I_1$. 
Let $p: P\ra \Ps 1$ be a non-trivial principal homogeneous space of $\pi_{a,b,c}$ given by an order $5$ element in the fiber $\pi_{a,b,c}^{-1}(0)$ (see \cite[Ch.V, \S 4]{CD}). Then  $P$ is a rational elliptic surface with a multiple smooth elliptic fiber of multiplicity $5$ over $0$.
Let $Z$ be the surface obtained by blowing up the intersection points of the component of multiplicity $3$ of $p^{-1}(\infty)$ with those of multiplicity $2$  and then blowing down the four connected components of the proper transform of $p^{-1}(\infty)$.
Then an easy computation shows that $Z$ is a log Enriques surface of index $5$ with $4$ singular points (the images of the components of $p^{-1}(\infty)$). 
Let $X$ be the canonical cover of $Z$ and $\sigma$ be a generator of the covering transformation group of the cover. By      \cite[Theorem  5.1]{Zhang2}, since $Z$ has three singular points of type $(5, 2)$ and one of type $(5,1)$, then  $X$ is a $K3$ surface. Moreover,  $\sigma$ is an order $5$ non-symplectic automorphism of $X$ with only isolated fixed points.
\end{ex}

\section{Order 7}
The local actions of an order $7$ non-symplectic  automorphism on a $K3$ surface at a fixed point are of four types 
$$ A_{7,0}=\left(\begin{array}{cc}
1 &0\\
0&\zeta_7
\end{array}
\right),\ A_{7,1}=\left(\begin{array}{cc}
\zeta_7^2 &0\\
0& \zeta_7^6
\end{array}
\right)
,\   A_{7,2}=\left(\begin{array}{cc}
\zeta_7^3 &0\\
0& \zeta_7^5
\end{array}
\right),\ 
A_{7,3}=\left(\begin{array}{cc}
\zeta_7^4 &0\\
0& \zeta_7^4
\end{array}
\right)
.$$

\begin{ex}\label{7fam1}
Let $X_{a,b}$ be the elliptic $K3$ surface with Weierstrass equation
$$y^2=x^3+(at^7+b)x+(t^7-1),\quad a,b\in \C.$$
For generic $a,b$ the fibration has one fiber of Kodaira type $III$ over $t=\infty$ and $21$ fibers of type $I_1$. 
 Observe that $X$  carries the order $7$ automorphism
\begin{equation}  \sigma(x,y,t)=(x,y,\zeta_7t) .\end{equation}
Observe that:
 \begin{enumerate}[1)]
\item If $b$ is generic, then  $X_{0,b}$ has one singular fiber of type $II^*$ over $t=\infty$ and  $14$ singular fibers of type $I_1$;
\item  if $a$ is generic and $b^3=-27/4$, then $X_{a,b}$ has one singular fiber of type $III$ over $t=\infty$, one of type $I_7$ over $t=0$ and $14 $ fibers of type $I_1$;
\item if  $b^3=-27/4$, then $X_{0,b}$ has one singular fiber of type $II^*$ over $t=\infty$, one of type $I_7$ over $t=0$ and $7$ of type $I_1$.
\end{enumerate}
\end{ex}

\begin{ex}\label{7fam2}
Let $\pi_{a,b}:Y_{a,b}\lra \Ps 1$ be the rational jacobian elliptic surface with Weierstrass equation
$$y^2=x^3+tx+(at+b),\ \ a,b\in \C.$$
For generic $a,b$ the elliptic fibration has a fiber of Kodaira type $III^*$ at $t=\infty$ and $3$ fibers of type $I_1$. 
 To this rational surface we can associate a $K3$ surface with an order seven non-symplectic automorphism having zero-dimensional fixed locus. The construction is similar to the one described in Example \ref{log5}.
Let $p: P\ra \Ps 1$ be a non-trivial principal homogeneous space of $\pi_{a,b}$ given by an order $7$ element in the fiber $\pi_{a,b}^{-1}(0)$. The surface $Z$ obtained by blowing up the intersection points of the component of multiplicity $4$ of $p^{-1}(\infty)$ with those of multiplicity $3$  and then blowing down the three connected components of the proper transform of $p^{-1}(\infty)$ is a log Enriques surface of index $7$ with $3$ singular points (the images of the components of $p^{-1}(\infty)$). 
By  \cite[Theorem  5.2]{Zhang2}  the canonical cover of $Z$ is a $K3$ surface and a generator of the covering transformation group is an order $7$ non-symplectic automorphism with only isolated fixed points.
 \end{ex}

\begin{table}
\begin{tabular}{c|c|c|c|c|c|c}
$n_1$ & $n_2$ & $n_3$ & $g$ & $k$&   $T(\sigma)$ & $S(\sigma)$\\
\hline
2  &1  &0  &1 & 0 &$U\oplus U\oplus E_8\oplus A_6$& $U\oplus K_7$\\ 
 2 &1 &0 &- &  -& $U(7)\oplus U\oplus E_8\oplus A_6$ & $U(7)\oplus K_7$\\
 \hline
 4& 3& 1 &1 & 1&$U\oplus U\oplus E_8$& $U\oplus E_8$\\ 
 4& 3&1 &0 & 0&$U(7)\oplus U\oplus E_8 $& $U(7)\oplus E_8$\\
\hline 
 6& 5& 2 &0 & 1&$U\oplus U\oplus K_7$ & $U\oplus E_8\oplus A_6$\\
  \end{tabular}\\
 \ \\
 \ \\
 \caption{Order 7}\label{ord7}
\end{table}
\begin{thm}\label{o7} The fixed locus of  a non-symplectic automorphism $\sigma$ of order $7$  of a $K3$ surface is either
\begin{itemize}
\item the union of two isolated points of type $1$ and one point of type $2$ if $S(\sigma)$ is isomorphic to $U(7)\oplus K_7$, or
\item the disjoint union of $n_i$ isolated points of type $i$, a smooth curve of genus $g$ and $k$ smooth rational curves where $(g,k,n_i)$ appears in a row of Table \ref{ord7}.
\end{itemize}
 The same table gives the corresponding invariant lattice $S(\sigma)$ and its orthogonal complement $T(\sigma)$.
\end{thm}
\proof 
The lattices in Table \ref{ord7} can be found by means of  Theorem \ref{RS} and Theorem \ref{ns}. 
For any such lattice $S(\sigma)$ the invariants $g,k, n_i$ of  $X^{\sigma}$ can be computed by means of Lemma \ref{chi}, Theorem \ref{point} and Corollary \ref{genus}.	\qed\\

We will now show that  a $K3$ surface $X$ with a non-symplectic automorphism $\sigma$ of order $7$ such that $S_X=S(\sigma)$ belongs to one of the two families in Examples \ref{7fam1}, \ref{7fam2}.  This will also show that all cases in Table \ref{ord7} do appear.

Observe that each lattice $S(\sigma)$ in the table contains a copy of either $U$ or $U(7)$. This implies that the generic $K3$ surface with Picard lattice isometric to $S$ has an elliptic fibration.
By Theorem \ref{o7} the invariant lattice of $\sigma$ is of the following types.\\

\noindent \fbox{$S(\sigma)\cong U\oplus K_7$}  A $K3$ surface $X$ with $S_X=S(\sigma)$ has a jacobian elliptic fibration with a unique reducible fiber of type $\tilde A_1$ (since $K_7$ contains a unique $(-2)$-curve, as can be checked directly). 
 Observe that $\sigma$ preserves the elliptic fibration and  its action on the basis of the fibration is not trivial, since otherwise a smooth fiber would have an order $7$ automorphism with a fixed point in the intersection with the section.
Thus $\sigma$ preserves two fibers: one of them is smooth and the other one is the reducible fiber.
By Theorem \ref{o7} we know that $\sigma$ fixes exactly $3$ isolated points and a smooth elliptic curve.  
This implies that  the reducible fiber  is of Kodaira type $III$ (if it were of type $I_2$ it should contain a fixed rational curve).  

Since $X$ has a jacobian elliptic fibration with an order $7$ automorphism acting non trivially on the basis, then we can write it in Weierstrass form:
$$y^2=x^3+f(t)x+g(t),$$
where $\sigma$ acts as $(x,y,t)\mapsto (x,y,\zeta_7t)$. The polynomials $f(t), g(t)$ are invariant for this action and $\deg(f)\leq 8$, $\deg(g)\leq 12$. Hence for a proper choice of coordinates we can assume $f(t)=(at^7+b)t^m$ and $g(t)=(ct^{7}-d)t^n$ where $m\leq 1$, $n\leq 5$.
Since the discriminant divisor is also $\sigma$-invariant and  there are exactly $21$ fibers of type $I_1$ in $\mathbb P^1\backslash \{\infty\}$, then $\Delta(t)=\delta(t^7-\alpha)(t^7-\beta)(t^7-\gamma)$.
Looking at the coefficients of the equality $\Delta(t)=4f(t)^3+27g(t)^2$ we can deduce that $m=n=0$, hence the Weierstrass equation of the $K3$ surface is of type
$$y^2=x^3+(at^7+b)x+(ct^{7}-d).$$
After applying a suitable change of coordinates we obtain a surface in the family of Example \ref{7fam1}.

Similar remarks  show that $K3$ surfaces with $S_X\cong  U\oplus E_8$, $U\oplus K_7\oplus A_6=U(7)\oplus E_8$ and $U\oplus E_8\oplus A_6$ belong to the families a), b) and c) respectively in Example \ref{7fam1}.\\

\noindent \fbox{$S(\sigma)\cong U(7)\oplus K_7$}\ 
By \cite[Corollary 3, \S 3]{PS}  a $K3$ surface $X$ with $S_X= S(\sigma)$ admits a $\sigma$-invariant elliptic fibration $\pi:X\ra \Ps 1$. 
  
By Theorem \ref{o7} we know that $\sigma$ has exactly $3$ fixed points. This implies that $\pi_1$ has at most one reducible fiber of type $III$. 
Observe that $\sigma$ induces an order $7$ automorphism on the basis of $\pi_ 1$ with two fixed points i.e.  $\sigma$ preserves two fibers. We can assume them to be a smooth fiber over $0$ and a fiber of type $III$ over $\infty$.  Moreover, $\pi_1$ has generically $21$ fibers of type $I_1$, divided in $3$ $\sigma$-orbits.

The quotient $X/\langle \sigma\rangle$ is a log Enriques surface of index $7$ (see Lemma 1.7 and its proof in \cite{OZ4}) and $\pi_1$ induces and elliptic fibration $\bar \pi_1:X/\langle\sigma\rangle\ra \Ps 1$. Let $Y$ be the minimal resolution of $X/\langle\sigma\rangle$. The proper transform of $\bar\pi_1^{-1}(\infty)$ is a $(-1)$-curve. After contracting this curve we get a minimal rational elliptic surface $\tilde\pi_1:\bar Y\ra\Ps 1$ with one fiber of type $7I_0$, one of type $III^*$ and three of type $I_1$. It can be easily proved that the jacobian fibration of $\tilde \pi_1$ is as in Example \ref{7fam2}.

 \begin{rem} A different model for the  $K3$ surface with $S_{X}\cong U\oplus E_{8}\oplus A_{6}$ was given by S. Kond\=o in \cite[\S 7, (7.5)]{Kondo1}: 
$$y^{2}=x^{3}+t^{3}x+t^{8}, \quad \sigma (x, y, t)=(\zeta ^{3}x, \zeta y, \zeta ^{2}t).$$
The singular fibers of this elliptic fibration are of type III$^{\ast }$, of type IV$^{\ast }$ and 7 of type I$_{1}$.
Moreover  it follows from \cite[$\S 5$]{Shioda}  that the rank of its Mordell-Weil group is $1$.
 \end{rem}

\section{Order 11}
 
Non-symplectic automorphisms of order $11$ of $K3$ surfaces have been classified in 1999 by K. Oguiso and D-Q. Zhang \cite{OZ1}. We provide here an alternative view of their classification.
\begin{ex}\label{11fam1}
Let $X_{a}$ be the elliptic $K3$ surface with Weierstrass equation
$$y^2=x^3+ax+(t^{11}-1),\quad a\in \C.$$
For generic $a\in \C$ the fibration has one fiber of Kodaira type $II$ over $t=\infty$ and $22$ fibers of type $I_1$. 
 Observe that $X$  carries the order $11$ automorphism
\begin{equation}  \sigma(x,y,t)=(x,y,\zeta_{11}t) .\end{equation}
 If $a^3=-27/4$ then $X_{a}$ has one singular fiber of type $II$ over $t=\infty$, of type $I_{11}$ over $t=0$ and  $11$ fibers  of type $I_1$.
 \end{ex}

\begin{ex}\label{11fam2} 
 
Let $\pi_{a}:Y_a\lra \PP^1$ be the rational jacobian elliptic surface with Weierstrass equation
$$
y^2=x^3+x+(t-a),\ \ a\in \C.
$$
For generic $a$ the family has a fiber of type $II^*$ over $t=\infty$ and two fibers of type $I_1$ over the zeros of $\Delta=4+27(t-a)^2$. As in Examples \ref{log5} and \ref{7fam2}, we associate to $Y_a$ a $K3$ surface with a non-symplectic automorphism with zero-dimensional fixed locus.
Let $p:P\ra\PP^1$ be a non-trivial principal homogeneous space of $\pi_{a}$  given by an element of order $11$ in the fiber $\pi_{a}^{-1}(0)$. The  surface $Z$, obtained  after blowing up the intersection point of the components of multiplicity $5$ and $6$ in $p^{-1}(\infty)$ and then blowing down the proper transform of the fiber $p^{-1}(\infty)$, is a log Enriques surface of index $11$ with two singular points. By \cite[Theorem 5.1]{Zhang2}  the canonical cover of $Z$ is a $K3$ surface and a generator of the Galois group of the covering is a non-symplectic order $11$ automorphism  with two isolated fixed points. \end{ex}
\begin{table}
\begin{tabular}{c|c|c|c|c|c|c|c}
$n_1=n_4$&   $n_2=n_3$ & $n_5$& $g$ & $k$&   $T(\sigma)$ & $S(\sigma)$\\
\hline
1& 0& 0 & 1&0 &$U\oplus U\oplus  E_8 \oplus E_8$ & $U$\\
 1  &0  &0  &- &-  &$U\oplus U(11)\oplus E_8\oplus E_8$& $U(11)$\\ 
  \hline 
 3  &2  &1  &0 & 0 &$ K_{11}(-1)\oplus E_8$& $U\oplus A_{10}$\\ 
 
  \end{tabular}\\

 \ \\

  \caption{Order 11}\label{ord11}
\end{table}
\begin{thm} \label{o11} The fixed locus of  a non-symplectic automorphism $\sigma$ of order $11$  of a $K3$ surface is either
\begin{itemize}
\item the union of one  point of type $1$ and one of type $4$ if $S(\sigma)\cong U(11)$, or
\item the disjoint union of $n_i$ isolated points of type $i$, a smooth curve of genus $g$ and $k$ smooth rational curves where $(g,k,n_i)$ appears in a row of Table \ref{ord11}.
\end{itemize}
 The same table gives the corresponding invariant lattice $S(\sigma)$ and its orthogonal complement $T(\sigma)$.
\end{thm}
\proof 
The lattices in Table \ref{ord11} can be found by means of  Theorem \ref{RS} and Theorem \ref{ns}. 
For any such lattice $S(\sigma)$ the invariants $g,k, n_i$ of $X^{\sigma}$ can be computed by means of Lemma \ref{chi}, Theorem \ref{point} and Corollary \ref{genus}.	\qed\\
 
 We will now show that  a $K3$ surface $X$ with a non-symplectic automorphism $\sigma$ of order $11$ such that $S_X=S(\sigma)$ belongs to one of the two families in Examples \ref{11fam1}, \ref{11fam2}.  This will also show that all cases in Table \ref{ord11} do appear.
Note that also in this case $X$ admits an elliptic fibration, since $S(\sigma)$ contains a copy of either $U$ or $U(11)$.\\

\noindent \fbox{$S(\sigma)\cong U$} A $K3$ surface $X$ with $S_X=S(\sigma)$ has a jacobian fibration with no reducible fibers. The order $11$ automorphism acts on the basis of the fibration with two fixed points i.e. it preserves two fibers. By Theorem \ref{o11} we know that the fixed locus contains one elliptic curve and two points. Thus there is one singular fiber  with two fixed points and a smooth fiber, which is pointwise fixed. 
Since $\sigma$ acts on the basis of the fibration and $\mathcal X(X)=24$, it follows that the $\sigma$-invariant singular fiber is of Kodaira type II.
Thus we are in Example \ref{11fam1}.\\

\noindent\fbox{$S(\sigma)\cong U\oplus A_{10}$} 
By Theorem \ref{o11} the fixed locus contains  one smooth rational curve $R$ and $11$ points. 
Note that $X$ in this case has a jacobian fibration with one fiber of type $I_{11}$ and $R$ is necessarily a component of this reducible fiber. Hence the fiber $I_{11}$ contains $9$ fixed points. The remaining two fixed points belong to an irreducible fiber of type $II$.  Thus we are in Example \ref{11fam1}.\\

\noindent \fbox{$S(\sigma)\cong U(11)$} By \cite[Corollary 3,\S 3]{PS} a $K3$ surface $X$ with $S_X=S(\sigma)$ carries  an elliptic fibration $\pi:X\ra\Ps 1$ (not jacobian). 
By Theorem \ref{o11} we know that there are no curves in the  fixed locus and two isolated points.   Hence $\pi$  has at most one reducible fiber of type $II$. Since $\sigma$ induces an order $11$ automorphism of $\PP^1$, then there are two fibers which are preserved. These are the fiber of type $II$ and a smooth elliptic fiber, which we can assume to be over $\infty$ and $0$. For the generic $K3$ the other singular fibers are $22$ of type $I_1$, divided in two orbits.

The quotient $X/\langle\sigma\rangle$ is a log Enriques surface of index $11$ (see Lemma 1.7 and its proof, \cite{OZ4}) with two singular points (the images of the fixed points on $X$) and an elliptic fibration $\bar{\pi}:X/\langle\sigma\rangle\ra \PP^1$. We can now consider the minimal resolution of $X/\langle\sigma\rangle$, where the proper transform of  $\bar{\pi}^{-1}(\infty)$ is a $(-1)$-curve. After contracting this $(-1)$-curve we obtain a smooth rational elliptic  surface $Y\ra\PP^1$ with   $11$ fibers of type $I_0$, one fiber of type $II^*$ and two fibers of type $I_1$. Thus we are in  Example \ref{11fam2}.

\begin{ex}[\cite{Kondo1}]
Let $X$ be the $K3$ surface  with non-symplectic automorphism:
$$y^2=x^3+t^5x+t^2,\ \sigma(x,y,t)= (\zeta_{11}^5x,\zeta_{11}^2y,\zeta_{11}^2t).$$
The elliptic fibration has one singular fiber of type $IV$, one of type $III^*$ and $11$ of type $I_1$. This is the only $K3$ surface with order $11$ automorphism such that $\rank T(\sigma)=10$.
Note that the fiber of type $III^*$ contains $7$ fixed points and a rational fixed curve, while the fiber of type $IV$ contains $4$ fixed points.
\end{ex}

\section{Order 13, 17, 19}
$K3$ surfaces with non-symplectic automorphisms of order $13, 17$ and $19$ are well known and studied in \cite{Kondo1} and  \cite{OZ2}.
The following examples are due to Kond\=o.
\begin{ex}\label{fam13}
Let $X$ be the $K3$ surface  with non-symplectic automorphism:
$$y^2=x^3+t^5x+t,\quad \sigma(x,y,t)= (\zeta_{13}^5x,\zeta_{13}y,\zeta_{13}^2t).$$
The elliptic fibration has one singular fiber of type $II$, one of type $III^*$ and $13$ fibers of type $I_1$.  
Note that the fiber of type $III^*$ contains $7$ fixed points and a rational fixed curve, while the fiber of type $II$ contains $2$ fixed points.
\end{ex}

\begin{ex}\label{fam17}
Let $X$ be the $K3$ surface with non-symplectic automorphism:
$$y^2=x^3+t^7x+t^2,\quad \sigma(x,y,t)= (\zeta_{17}^7x,\zeta_{17}^2y,\zeta_{17}^2t).$$
The elliptic fibration has one singular fiber of type $IV$, one of type $III$ and $17$ fibers of type $I_1$.  
Note that the fiber of type $IV$ contains $4$ fixed points, while the fiber of type $III$ contains $3$ fixed points.
\end{ex}
\begin{ex}\label{fam19}
Let $X$ be the $K3$ surface  with non-symplectic automorphism:
$$y^2=x^3+t^7x+t,\quad \sigma(x,y,t)= (\zeta_{19}^7x,\zeta_{19}y,\zeta_{19}^2t).$$
The elliptic fibration has one singular fiber of type $II$, one of type $III$ and $19$ fibers of type $I_1$.  
Note that the fiber of type $II$ contains $2$ fixed points, while the fiber of type $III$ contains $3$ fixed points.
\end{ex}

\begin{table}
\begin{tabular}{c|c|c|c|c|c|c}
$n_1=n_2$ &  $n_3$ & $n_4$& $n_5=n_6$& $k$ &   $T(\sigma)$ & $S(\sigma)$\\
 \hline 
 3 & 2& 1&  0 &  1&  $U\oplus H_{13}\oplus E_8$ & $H_{13}\oplus E_8$
 
  \end{tabular}\\

 \ \\

 \ \\

  \caption{Order 13}\label{t4}
\end{table}
\begin{table}
\begin{tabular}{c|c|c|c|c|c|c}
$n_1=n_2=n_3=n_4$& $n_5=n_8$&$n_6$ & $n_7$ &     $k$ &    $T(\sigma)$ & $S(\sigma)$\\
\hline
0& 1& 2 & 3    &  0 &    $U\oplus U\oplus E_8\oplus L_{17}$& $U\oplus L_{17}$ \\ 
 \end{tabular}\\

 \ \\

 \ \\

\caption{Order 17}\label{t5}
\end{table}
\begin{table}

\begin{tabular}{c|c|c|c|c|c}
$n_1=n_2=n_3=n_8=n_{9}$ & $n_4=n_{6}=n_7$ &  $n_5$ &   $k$ &    $T(\sigma)$ & $S(\sigma)$\\
\hline
$0$  & $1$ & $2$ & $0$ &  $K_{19}(-1)\oplus E_8\oplus E_8$ & $U\oplus K_{19}$
   \end{tabular}\\

 \ \\

 \ \\
\caption{Order 19}\label{t6}
\end{table}

\begin{thm} \label{o131719}  A $K3$ surface with a non-symplectic automorphism $\sigma$ of order $13, 17$ or $19$  
 is isomorphic to the surface in Example \ref{fam13},  \ref{fam17} or  \ref{fam19} respectively.
 
The fixed locus of such automorphism  is the union of $n_i$ points of type $i$ and $k$ smooth rational curves, as described in Tables \ref{t4}, \ref{t5} and \ref{t6} respectively.
The same table gives the corresponding invariant lattices $S(\sigma)$ and their orthogonal complements.\end{thm}
 \proof
 The lattices in tables  \ref{t4}, \ref{t5} and \ref{t6}  can be found by means of  Theorem \ref{RS} and Theorem \ref{ns}. 
For any such lattice $S(\sigma)$ the invariants $g,k, n_i$ of the fixed locus $X^{\sigma}$ can be computed by means of Lemma \ref{chi}, Theorem \ref{point} and Corollary \ref{genus}.	 

We now show that a K3 surface with a non-symplectic automorphism of order $13, 17$ or $19$ is as in Example \ref{fam13},  \ref{fam17} or  \ref{fam19} respectively.

If $p=13$ then $S(\sigma)\cong H_{13}\oplus E_8$. Let $e_1,e_2$ be the generators of $H_{13}$ with intersection matrix as in section 1 and $f\in E_8$. The vectors $e_1-e_2+f, e_2+f $ and a basis of $f^{\perp}\cong E_7\subset E_8$ generate a primitive sublattice $S$ of $S(\sigma)$ isometric to $U\oplus E_7$ such that $S^{\perp}$ contains no $(-2)$-curves.
It follows that $X$ admits a jacobian elliptic fibration $\pi$ with a unique reducible fiber $F$ of type $III^*$. 
Observe that $\sigma$ induces a non-trivially action on the basis of $\pi$, since otherwise the general fiber  would have an order $13$ automorphism with a fixed point (the intersection with a section of $\pi$).
Thus $\sigma$ preserves exactly two fibers of $\pi$.
Since $\mathcal X(X)=24$ and $\mathcal X(F)=9$, then $\pi$ has also a $\sigma$-orbit of $13$ singular fibers of type $I_1$ and  a $\sigma$-invariant fiber of type $II$.
 Working as in the proof of Theorem \ref{o7} it can be proved that there is only one jacobian fibration with this property (see also \cite[\S 4]{OZ2}). Thus $X$ is isomorphic to the surface in Example \ref{fam13}. 

The proofs for $p=17, 19$ are similar. In these cases $S(\sigma)$ contains a primitive sublattice $S\cong U\oplus A_2\oplus A_1$ and  $S\cong U\oplus A_1$ respectively such that $S^{\perp}$ contains no $(-2)$-curves. Thus the surface admits a jacobian fibration with reducible fibers of types $\tilde A_2\oplus \tilde A_1$ and $\tilde A_1$ respectively.
This implies, together with the fact that $\sigma$ acts non-trivially on the basis of the fibration and a computation of $\mathcal X(X)$, that $X$ is isomorphic to either the surface in Example \ref{fam17} or \ref{fam19}.\qed

\section{Moduli spaces}\label{modulisection}
 Let $\rho\in \is(L_{K3})$ be an isometry of prime order $p$ with hyperbolic invariant lattice 
$$S(\rho)=\{x\in L_{K3}:\ \rho(x)=x\}$$ 
and let $[\rho]$ be its conjugacy class in $\is(L_{K3})$. A $[\rho]$-\emph{polarized} $K3$ surface is a pair $(X,\sigma)$ where $X$ is a $K3$ surface and $\sigma$ a non-symplectic automorphism of $X$ of order $p$ such that 
$$\sigma^*(\omega_X)=\zeta_p\omega_X\quad \mbox{ and } \quad \sigma^*=\phi\circ \rho \circ \phi^{-1}$$ for some isometry $\phi:L_{K3}\ra H^2(X,\Z)$, which is called a \emph{marking}.
Two $[\rho]$-polarized $K3$ surfaces $(X,\sigma),(X',\sigma')$ are \emph{isomorphic} if there exists an isomorphism $f:X\ra X'$ such that $f^{-1}\circ \sigma' \circ f=\sigma$.

Observe that a $[\rho]$-\emph{polarized} $K3$ surface is algebraic by \cite[Theorem 3.1]{N1}. 
If $h$ is an ample class in $S_X$, then the average $\sum_{i=1}^p (\sigma^*)^i(h)$ is an ample class in $S(\sigma^*)$.
This implies that any marked $[\rho]$-polarized $K3$ surface is $S(\rho)$-ample polarized  (see \cite[\S 10]{DK})

A moduli space for such polarized surfaces can be constructed as follows (see \cite[\S 11]{DK}). 
Let $S(\rho)^{\perp}=T(\rho)$ and $V^{\rho}=\{x\in L_{K3}\otimes \C:\ \rho_{\C}(x)=\zeta_p x\}\subset T(\rho)\otimes \C$ be an eigenspace of the natural extension of  $\rho$ to  $L_{K3}\otimes \C$.  Consider the space
$$D^{\rho}=\{w\in \mathbb P(V^{\rho}):\ (w,\bar w)>0, (w,w)=0\}.$$
This is a type IV Hermitian symmetric space of dimension $r(T(\rho))-2$ if $p=2$ and  it is isomorphic to a complex ball of dimension $r(T(\rho))/(p-1)-1$ if $p>2$ (note that if $\zeta_p\notin\R$ then the condition $(w,w)=0$ is automatically true).

\noindent Furthermore, consider the divisor
$$\Delta^{\rho}=\bigcup_{\delta\in T(\rho),\\ \delta^2=-2} D^{\rho}\cap \delta^{\perp}$$
and the discrete group $\Gamma^{\rho}=\{\gamma\in \is(L_{K3}):\  \gamma \circ \rho=\rho \circ \gamma\}.$

\begin{thm}\label{moduli}
The orbit space $\mathcal M^{\rho}:=\Gamma^{\rho}\backslash (D^{\rho}\backslash \Delta^{\rho})$ parametrizes isomorphism classes of $[\rho]$-polarized $K3$ surfaces.
\end{thm}
\proof 
 Let $(X,\sigma)$ be a $[\rho]$-polarized $K3$ surface and $\phi:L_{K3}\ra H^2(X,\Z)$ be an isometry such that $\sigma^*=\phi\circ \rho \circ \phi^{-1}$. Since $\sigma^*(\omega_X)=\zeta_p\omega_X$, then $\ell:=\phi_{\C}^{-1}(\C\omega_X)\in D^{\rho}$.  If $\ell \in \delta^{\perp}$ for some $\delta\in T(\rho),  \delta^{2}=-2$, then either $\phi(\delta)$ or  $\phi(-\delta)$ would be an effective class $x$ such that $\sigma^*(x)+\cdots+(\sigma^*)^{p-1}(x)=-x$. This gives a contradiction since the left side is an effective divisor and the right side is not, thus $\ell\not\in \Delta^{\rho}$.
An isometry $\phi'$ also satisfies $\sigma^*=\phi'\circ \rho \circ \phi'^{-1}$ if and only if $\phi^{-1}\circ\phi'\in \Gamma^{\rho}$. It is easy to check that two isomorphic $[\rho]$-polarized $K3$ surfaces give the same point $\ell\in D^{\rho}$ modulo $\Gamma^{\rho}$.

 Conversely, let  $\ell\in D^{\rho}\backslash \Delta^{\rho}$. By the surjectivity theorem of the period map \cite[\S X]{Ge} and \cite[Theorem 10.1]{DK} there exist a $K3$ surface and a marking $\phi:L_{K3}\ra H^2(X,\Z)$  such that $\phi_{\C}(\ell)=\C\omega_X$ and $\phi_{|S(\rho)}:S(\rho)\ra S_X$ is an ample polarization (see \cite[\S 10]{DK}). 
 Let $\psi=\phi\circ \rho\circ \phi^{-1}$, then $\psi_{\C}(\ell)=\ell$ and its invariant lattice $S(\psi)=\phi(S(\rho))$  contains an ample class .
  Moreover,  $S(\psi)^{\perp}\cap \C\omega_X^{\perp}$ contains no elements of self-intersection $-2$ since $\ell\not\in \Delta^{\rho}$. Thus, by the global Torelli Theorem \cite[\S IX]{Ge} there is a unique automorphism $\sigma$ of $X$ such that
 $\sigma^*=\psi$.
 It is clear that $\sigma$ is non-symplectic of order $p$ and that $\sigma^*(\omega_X)=\zeta_p \omega_X$. \qed

\begin{rem}
In \cite[Theorem 11.3]{DK} it is proved that $\Gamma_0^{\rho}\backslash (D^{\rho}\backslash \Delta^{\rho})$, where $\Gamma_0^{\rho}=\{\gamma \in \Gamma^{\rho}: \gamma_{|S(\rho)}=id\}$, parametrizes isomorphism classes of  $[\rho]$-polarized $K3$ surfaces with the extra data of an ample polarization  $j:S(\rho)\ra S_X$ (see also \cite[Remark 11.4]{DK}).
\end{rem}

 The following result says when two $K3$ surfaces with non-symplectic automorphisms of order $p$ belong to the same  moduli space $\mathcal M^{\rho}$.
\begin{pro}\label{uni}
Two pairs $(X_1,\sigma_1)$, $(X_2,\sigma_2)$ of $K3$ surfaces with non-symplectic automorphisms of order $p$ 
are polarized by the same $\rho\in O(L_{K_3})$ if and only if $S(\sigma_1)\cong S(\sigma_2)$.
This is also equivalent to say that $X_1^{\sigma_1}$ is homeomorphic to $X_2^{\sigma_2}$ for $p>2$.
\end{pro}
\proof
It is clear that two $[\rho]$-polarized $K3$ surfaces have $S(\sigma_1)\cong S(\rho)\cong S(\sigma_2)$.
In sections \S\,5,6,7,8 we proved that a $K3$ surface with a non-symplectic automorphism $\sigma$ of order $p=5,7,11,13,17,19$ and given invariant lattice $S(\sigma)=S_X$ belongs to an irreducible family.
A similar result holds for $p=2,3$ by \cite[\S 4]{N2} and \cite[\S 5]{AS}. This implies that two pairs with $S(\sigma_1)\cong S(\sigma_2)$ belong to the same irreducible component $\mathcal M^{\rho}$.
The last statement follows from Remark \ref{obs}.
\qed

\begin{rem}
In \cite{N1} V.V. Nikulin proved that the action of a symplectic automorphism on the $K3$ lattice (up to conjugacy) only depends on the number of its fixed points.
Proposition \ref{uni} gives a similar statement for non-symplectic automorphisms of prime order $p>2$.
\end{rem}

By Theorem \ref{moduli} the moduli space of $[\rho]$-polarized $K3$ surfaces is an irreducible quasi-projective variety.
 Moreover, as proved in \cite[Theorem 11.7]{DK}, any point $\ell\in D^{\rho}$ is the period point of some $[\rho']$-polarized $K3$ surface, where $\rho'$ has order $p$ and $[\rho']$ is not equal to $[\rho]$ if $\ell\in \Delta^{\rho}$.
Thus, the quotient $\Gamma^{\rho}\backslash D^{\rho}$ is an irreducible subvariety of the moduli space $\mathcal M_{K3}^p$ of $K3$ surfaces with a non-symplectic automorphism of order $p$.
We are interested in identifying the maximal irreducible subvarieties of this type in $\mathcal M_{K3}^p$, or equivalently, its irreducible components.

 \begin{thm}\label{moduli1}
The following table gives the  number $\#$ of irreducible components of the moduli space $\mathcal M_{K3}^p$,  their dimensions and the Picard lattice $S(\sigma)$ of the generic $K3$ surface in each component, for any prime $p$.
\begin{center}
\begin{table}[h]
\begin{tabular}{cccc}
$p$ & \# & $\dim$ & $S(\sigma)$\\
\hline
$2$ & $2$ & $19, 18$ & $(2),\ U(2)$\\
$3$ & $3$ &  $9,9,6$ & $ U,\ U(3),\ U(3)\oplus E_6^*(3)$\\
$5$& $2$ & $4,3$ & $H_5,\ H_5\oplus A_4^*(5)$\\
$7$ & $2$ & $2,2$ & $U\oplus K_7,\ U(7)\oplus K_7$\\ 
$11$ & $2$ & $1,1$ & $U,\ U(11) $\\
$13$ & $1$ & $0$ & $H_{13}\oplus E_8$\\
$17$ & $1$ & $0$ & $U\oplus L_{17}$\\
$19$& $1$ & $0$ & $U\oplus K_{19}$
\end{tabular}\\
\ \\
 \ \\\caption{Irreducible components of $\mathcal M_{K3}^p$}\label{mod}
\end{table}
\end{center}
\end{thm}
\vspace{-0.9cm} 

\proof 
 Observe that the moduli space $\mathcal M^{\rho}$ is in the closure of $\mathcal M^{\rho_2}$ if and only if there is $\rho_1\in[\rho]$ such that $D^{\rho_1}\subset D^{\rho_2}$, i.e. $V^{\rho_1} \subset V^{\rho_2}$.
This is equivalent to say that $T(\rho_1)\subset T(\rho_2)$ and $\rho_2 =\rho_1$ on $T(\rho_1)$.

If $p=2$, then $\rho_i=-id$ on $T(\rho_i)$, hence $\mathcal M^{\rho}\subset \mathcal M^{\rho_2}$ if and only if $T(\rho_1)\subset T(\rho_2)$, or equivalently $S(\rho_1)\supset S(\rho_2)$.  As a consequence of Theorem \ref{o2}, the  invariant lattice $S(\sigma)$ of a non-symplectic involution contains a primitive sublattice which is isometric to either $(2)$ or $U(2)$.
This implies, since $(2)$ clearly is not a sublattice of $U(2)$, that $\mathcal M_{K3}^2$ has two irreducible components whose generic elements are K3 surfaces with $S_X$ isomorphic to $(2)$ and  $U(2)$  respectively.  

The case $p=3$ is  \cite[Theorem 5.6]{AS}.
In Theorem \ref{o5} we proved that a $K3$ surface with an order $5$ non-symplectic automorphism either belongs to the family in Example \ref{fam1} or to the one in Example \ref{fam2}. These two families are irreducible and of dimensions  $4$ and  $3$ respectively.
Thus we only need to prove that the second component is not contained in the first one.
Assume that the generic pair $(X,\sigma_1)$ in the first family (we assume $\sigma^*_1 (\omega_X)=\zeta_5\omega_X$), with  $S_X=S(\sigma_1)\cong H_5\oplus A_4^*(5)$, also belongs to the second family. Since the orthogonal complement of $H_5$ in $S_X$ contains no ($-2$)-curves, then by Theorem \ref{moduli}, $X$ carries  an automorphism $\sigma_2$ of order $5$ such that $\sigma^*_2 (\omega_X)=\zeta_5\omega_X$ and $S(\sigma_2)\cong H_5$.
If $h\in S(\sigma_1)$, $h^2=2$ is as in \S 5, then the associated morphism  is a double cover of $\Ps 2$ branched along a smooth sextic curve $C$ (since $A_4^*(5)$ contains no ($-2$)-curves). Since $h$ is fixed by  $\sigma_i^*$, $i=1,2$, then $\sigma_i$ induces a projectivity $\bar \sigma_i$ of $\Ps 2$ which preserves $C$. The automorphism group of $C$ is finite, thus $\bar \sigma_1^4\circ \bar\sigma_2$ has finite order and  $ \sigma_1^4\circ \sigma_2$ is a symplectic automorphism of finite order of $X$.
By \cite{N1} this would imply that the Picard lattice of $X$ has rank $>8$, giving a contradiction.

If $p=7$, then in Theorem \ref{o7} we proved that a $K3$ surface with an order $7$ non-symplectic automorphism either belongs to the family in Example \ref{7fam1} or to the one in Example \ref{7fam2}.
Both are clearly irreducible of dimension $2$, thus they are the irreducible components of $\mathcal M_{K3}^7$.

If $p=11$, then in Theorem \ref{o11} we proved that a $K3$ surface with an order $11$ non-symplectic automorphism either belongs to the family in Example \ref{11fam1} or to the one in Example \ref{11fam2}.
Both are clearly irreducible and $1$-dimensional, thus they are the irreducible components of $\mathcal M_{K3}^{11}$.

If $p=13, 17$ or $19$ then by Theorem \ref{o131719} the moduli space $\mathcal M_{K3}^p$ is irreducible and $0$-dimensional.
\qed

 \begin{rem}
The general members of the two irreducible components of $\mathcal M_{K3}^2$ are double covers of the plane branched along a smooth sextic curve (if $S_X\cong (2)$) and double covers of a quadric branched along a smooth curve of bidegree $(4,4)$ (if $S_X\cong U(2)$).
Projective models for the general members of  $\mathcal M_{K3}^3$ are described in \cite{AS} and \cite{Taki}.
\end{rem}

\bibliographystyle{plain}

\begin{thebibliography}{10}
\bibitem{AN} V.~Alexeev, V.V.~Nikulin.
\newblock{Del {P}ezzo and {$K3$} surfaces}.
  \newblock MSJ Memoirs, Vol. 15, Mathematical Society of Japan, Tokyo, 2006.
    
\bibitem{AS} M.~Artebani, A.~Sarti.
\newblock{Non symplectic automorphisms of order $3$ on $K3$ surfaces.}
\newblock Math. Ann. (2008), 342, 903--921. 

\bibitem{AS2}
M.F.~Atiyah, I.M.~Singer.
\newblock{The index of elliptic operators: III.}
\newblock Ann. of Math., 87 (1968), 546--604.

\bibitem{BPV} W.~Barth, C.~Peters, A.~van de Ven.
\newblock{Compact complex surfaces. Springer, Berlin, Heidelberg, New York, 1984.}
 
 \bibitem{Ge} A.~Beauville, J.-P.~Bourguignon, M.~Demazure.
\newblock{G\'eometrie des surfaces K3: modules et p\'eriodes}. 
\newblock Ast\'erisque, vol. 126, Soc. Math. France, 1985.

 \bibitem{Bredon} G.E.~Bredon.
\newblock{Introduction to compact transformation groups}. 
\newblock Pure and Applied Math. 46, 
Academic Press, New York-London, 1972. 

\bibitem{CD} F.~Cossec, I.~Dolgachev.
\newblock{Enriques Surfaces I.}
\newblock Progress in Mathematics, 76. Birkh\"auser Boston, Inc., Boston, MA, 1989. 


\bibitem{DK} I.~Dolgachev, S.~Kond\=o.
\newblock{Moduli spaces of $K3$ surfaces and complex ball quotients}.
\newblock Arithmetic and geometry around hypergeometric functions, 43--100, 
Progr. Math., 260, Birkh\"auser, Basel, 2007. 



\bibitem{Gr} M.~J. Greenberg.
\newblock{Lectures on algebraic topology}. 
\newblock W. A. Benjamin, Inc., New York-Amsterdam, 1967.


\bibitem{Kh} V.M.~Kharlamov.
\newblock{The topological type of nonsingular surfaces in {$\mathbb P^3\mathbb R$}  of degree four}. 
\newblock Funct. Anal. Appl.,10:4 (1976) 295Ð-304. 

\bibitem{Kondo1}
S.~Kond\=o. 
\newblock{Automorphisms of algebraic $K3$ surfaces which act trivially on Picard groups}. 
\newblock J. Math. Soc. Japan, 44 (1992), 75--98.

\bibitem{Kondo2}
S.~Kond\=o.
\newblock{The moduli space of 5 points on $\mathbb{P}^{1}$ and $K3$ surfaces}. 
\newblock Progress in Mathematics, 260 (2007), 189--206. 

\bibitem{K} S.~Kond\=o.
\newblock{The moduli space of curves of genus $4$ and Deligne-Mostow's complex reflection groups}.
\newblock Advanced Studies in Pure Math., 36, Algebraic Geometry 2000, Azumino, 383--400.



  \bibitem{MO}
 N.~Machida, K.~Oguiso.
 \newblock{On $K3$ surfaces admitting finite non-symplectic group actions}.
              \newblock { J. Math. Sci. Univ. Tokyo,} 5, n.2 (1998), 273--297.


\bibitem{Milnor}
J.~Milnor.
\newblock{On simply connected $4$-manifolds}.
\newblock{International symposium in algebraic topology,} (1958), Universidad Nacional Autonoma de M\'exico and UNESCO, Mexico city, 122--128.

\bibitem{Mi} R.~Miranda.
\newblock{The basic theory of elliptic surfaces}.
\newblock  Dottorato di ricerca in Matematica, ETS Editrice, Pisa (1989).



\bibitem{Na}
Y.~Namikawa.
\newblock{Periods of Eriques surfaces}.
\newblock Math. Ann. 270 (1985), 201--222.


\bibitem{N1}
V.V.~Nikulin.
\newblock{Finite groups of automorphisms of K\"{a}hlerian surfaces of type $K3$}.
\newblock Moscow Math. Soc.,  38 (1980), 71--137.



\bibitem{N2}
V.V.~Nikulin.
\newblock{Integral symmetric bilinear forms and some of their applications}.
\newblock Math. USSR Izv.,  14 (1980), 103--167.



\bibitem{N3}
V.V.~Nikulin.
\newblock{Factor groups of groups of the automorphisms of hyperbolic forms with respect to subgroups generated by 2-reflections}.
\newblock  Soviet Math. Dokl.,  20 (1979), 1156--1158.


\bibitem{N4}
V.V.~Nikulin.
\newblock{Discrete reflection groups in Lobachevsky spaces and algebraic surfaces}.
\newblock Proceedings of the International Congress of Mathematicians, Vol. 1, 2 (Berkeley, Calif., 1986),  654--671, Amer. Math. Soc., Providence, RI, 1987.

\bibitem{OZ1}
K.~Oguiso, D-Q.~Zhang.
\newblock{$K3$ surfaces with order 11 automorphisms}.
\newblock \href{http://arxiv.org/pdf/math/9907020}{arXiv:math/9907020v1}.

\bibitem{OZ2}
K.~Oguiso, D-Q.~Zhang.
\newblock{On Vorontsov's theorem on $K3$ surfaces with non-symplectic group actions}.
\newblock Proc. Amer. Math. Soc.  128 (2000), no. 6, 1571-1580.

\bibitem{OZ3}
K.~Oguiso, D-Q.~Zhang.
\newblock{$K3$ surfaces with order five automorphisms}.
\newblock J. Math. Kyoto Univ. (1998) 419-438.

\bibitem{OZ4}
K.~Oguiso, D-Q.~Zhang.
\newblock{On extremal log Enriques surfaces, II}.
\newblock  Tohoku Math. J. 50 (1998), 419-436.

\bibitem{PS}
I.I.~Pjateckii-Shapiro, I.R.~Shafarevich.
\newblock{A Torelli Theorem for algebraic surfaces of type $K3$}.
\newblock Izv. Akad. Nauk SSSR Ser. Mat.   35  (1971) 530--572.
\newblock Math. USSR Izvestija Vol. 5 (1971) 547--588.

\bibitem{RS}
A.N.~Rudakov, I.~Shafarevich.
\newblock{Surfaces of type $K3$ over fields of finite characteristic.}
\newblock In: I. Shafarevich, Collected mathematical papers, Springer, Berlin (1989), 657--714.

\bibitem{Serre} J.-P.~Serre.
\newblock {A course in arithmetic.}
\newblock Graduate Texts in Mathematics, No.7, Springer-Verlag, New York-Heidelberg, 1973.

\bibitem{Shioda} T.~Shioda.
\newblock{An explicit algorithm for computing the Picard number of certain algebraic $K3$ surfaces.}
\newblock American Journal of Mathematics, 108 (1986), 415--432.

\bibitem{lat} N.J.A.~Sloane, AT\&T Labs-Research and G.~Nebe.
 \newblock{Catalogue of Lattices.} 
 \newblock University of Ulm,  \href{http://www.research.att.com/~njas/lattices}{http://www.research.att.com/${}_{\textrm{\symbol{126}}}$njas/lattices }.
 
\bibitem{Taki} S.~Taki.
\newblock{Classification of non-symplectic automorphisms of order 3 on $K3$ surfaces.}
\newblock{To appear in Math. Nachr. (2008).}


\bibitem{V}S.P.~Vorontsov.
\newblock{Automorphisms of even lattices arising in connection with automorphisms of algebraic $K3$-surfaces.} 
Vestnik Moskov. Univ. Ser. I Mat. Mekh. (1983), no. 2, 19--21. 

\bibitem{Zhang} D.-Q.~Zhang.
\newblock{Quotients of $K3$ surfaces modulo involutions.}
\newblock{ Japan. J. Math. (N.S.)  24  (1998),  no. 2, 335--366.}


\bibitem{Zhang1} D.-Q.~Zhang.
\newblock{Normal Algebraic Surfaces with trivial tricanonical divisors.}
\newblock{Publ. RIMS, Kyoto Univ. 33 (1997), 427--442.}

\bibitem{Zhang2} D.-Q.~Zhang.
\newblock{Logarithmic Enriques surfaces.}
\newblock{J. Math. Kyoto Univ. 31--2 (1991), 419--466.}

\bibitem{Zhang3} D.-Q.~Zhang.
\newblock{Normal Logarithmic Enriques surfaces, II.}
\newblock{J. Math. Kyoto Univ. 33--2 (1993), 357--397.}


 

\end{thebibliography}

\smallskip

\begin{thebibliography}{AMR}

\bibitem[K1]{K1} S.\ Kond${\rm \bar o}$, {\it A complex hyperbolic
structure of the moduli space of curves of genus three},  J. reine angew. Math., {\bf 525} (2000), 219--232.

\bibitem[K2]{K2} S.\ Kond$\bar {\rm o}$, {\it The moduli space of 5 points on $\bbP^1$ and $K3$ surfaces}, 
Progress in Mathematics, vol. {\bf 260} (2007), 189--206.

\bibitem[N]{N} I.\ Naruki,  {\it On a $K3$ surface which is a ball quotient},
Max Planck Institute Preprint Series.


\bibitem[RS]{RS1}
A.N.~Rudakov and I.\ Shafarevich,
{\it Surfaces of type $K3$ over fields of finite characteristic}, 
I. Shafarevich, Collected mathematical papers, Springer, Berlin (1989), 657--714.

\end{thebibliography}

\vskip 25pt

\noindent
{\small Michela Artebani, Departamento de Matem\'atica, Universidad de Concepci\'on, Casilla 160-C, Concepci\'on, Chile. e-mail: {\tt martebani@udec.cl}}.

\vskip 10pt

\noindent 
{\small Alessandra Sarti, Universit\'e de Poitiers, 
Laboratoire de Math\'ematiques et Applications, 
 T\'el\'eport 2 
Boulevard Marie et Pierre Curie
 BP 30179,
86962 Futuroscope Chasseneuil Cedex,
France. e-mail {\tt sarti@math.univ-poitiers.fr},\\ URL {\tt http://www.mathematik.uni-mainz.de/$\sim$sarti}}.

\vskip 10pt

\noindent 
{\small Shingo Taki, Graduate School of Mathematics, 
 Nagoya University, Chikusa-ku Nagoya 464-8602 Japan. e-mail {\tt m04022x@math.nagoya-u.ac.jp}}.

\newpage

%%%%%%%%%%%%%%%appendix kondo%%%%%%%%%%%%%%%%%%%%
\setcounter{section}{1}
\setcounter{thm}{0}
\begin{center}
{\bf Appendix: On Naruki's $K3$ surface}\\
\vspace{5mm}
{\rm{\small Shigeyuki Kond{$\bar{\rm o}$}}}\footnote{Research of the author is partially supported by
Grant-in-Aid for Scientific Research A-18204001 and Houga-20654001, Japan}
\end{center}

\vskip 10pt

As a moduli space of some $K3$ surfaces with a non-symplectic automorphism of order 5,
the quintic del Pezzo surface appears (\cite{K2}, see the following Remark \ref{re} for more details).
In this appendix we shall give a similar example in case of $K3$ surfaces with a non-symplectic 
automorphism of order $7$.

\subsection{Naruki's $K3$ surface}

Let $\zeta = e^{2\pi \sqrt{-1}/7}$.  We introduce a hermitian form of signature $(1,2)$ with variables $z=(z_1,z_2,z_3)$ by setting
$$H(z) = (\zeta + {\bar \zeta})z_1{\bar z}_1- z_2{\bar z}_2 - z_3{\bar z}_3.$$
We denote by $SU(1,2)$ the group of $(3,3)$-matrices of determinant 1 which are unitary with respect to $H$. The group
$SU(1,2)$ naturally acts on the complex ball of dimension 2
$$D= \{ (z_1,z_2,z_3) \in \bbP^2 : H(z) > 0\}.$$
We denote by $\Gamma$ the subgroup of $SU(1,2)$ consisting of elements whose entries are integers in $\bbQ (\zeta)$.
It is known that $\Gamma$ acts on $D$ properly discontinuously and the quotient $D/\Gamma$ is compact.  
We further denote by $\Gamma'$ the subgroup of $\Gamma$ consisting of matrices which are congruent to the identity matrix
modulo the principal ideal $P$ generated by $1-\zeta$.  Naruki \cite{N} showed that the quotient
$D/\Gamma'$ is isomorphic to a $K3$ surface $X$.

\subsection{$K3$ surfaces with a non-symplectic automorphism of order 7}
In the following we shall show that the Naruki's $K3$ surface $X$ is the moduli space of pairs of  a $K3$ surface and a non-symplectic 
automorphism of order 7.

Let $S = U(7) \oplus K_7$ and its orthogonal complement $T= U\oplus U(7)\oplus E_8\oplus A_6$ in $L_{K3}$ (see Table \ref{ord7}).

\begin{rem-def} \label{isometry}
By \cite[Theorem 2.1]{RS1} it follows that $T$ is isomorphic to $T'=U\oplus U\oplus K_7\oplus A_6^2$. An order $7$ isometry without non-zero fixed vectors on $T$ can be thus explicitely described as follows.

Let $U_1, U_2$ be two copies of the hyperbolic plane $U$ and
let $e_i, f_i$ be a basis of $U_i$, $(i=1,2)$ satisfying $e_i^2 =
f_i^2 = 0,\ (e_i, f_i) = 1$.
Let $x, y$ be a basis of $K_7$
satisfying $x^2 = -2,\ y^2 = -4,\ (x,y) = 1$.
Let $\rho_0$ be the isometry of $U_1\oplus U_2 \oplus
K_7$
defined by
$$\begin{array}{ll}
\rho_0(e_1) = e_1+f_1+e_2+f_2-y &
\rho_0(f_1) = 2e_1+e_2+2f_2-y\\
\rho_0(e_2) = -f_1+e_2+f_2+x &
\rho_0(f_2) = -f_1+e_2\\
\rho_0(x) = e_1+2f_1-e_2+f_2-x-y, &
\rho_0(y)= 3e_1-f_1+4e_2+3f_2+x-2y.
\end{array}$$
It is easy to see that $\rho_0$ has order $7$ and acts trivially on the
discriminant group of $U_1\oplus U_2 \oplus K_7$.

\noindent  An easy calculation shows that
$$v = (-1+\zeta^2+\zeta^4-\zeta^5)e_1+
(\zeta^3-1)f_1+(\zeta-\zeta^5)e_2+
(\zeta^2-\zeta^5)f_2+x+(1+\zeta^5) y$$
is an eigenvector of $\rho_0$ with the eigenvalue $\zeta$
and
$$(v, \bar{v}) = 7
(\zeta + \zeta^6).$$
On the other hand, let $r_1, \dots, r_6$ be a basis of
$A_6$ such that $r_i^2 = -2,\ (r_i, r_{i+1})= 1$
and the other $r_i$'s and $r_j$'s are orthogonal.
Consider the isometry of $A_6$ defined by
$$ \rho_6(r_i)=r_{i+1}, \ (1 \leq i \leq 5),
\quad \rho_6(r_6)
= -(r_1+r_2+r_3+r_4+r_5+r_6).$$
It is easy to see that $\rho_6$ acts trivially on the
discriminant group of $A_6$ and that 
$$w = r_1+(\zeta^6+1)r_2+(1+\zeta^5+\zeta^6)
r_3-
(\zeta+\zeta^2+\zeta^3)r_4-(\zeta+
\zeta^2)r_5-\zeta r_6$$
is an eigenvector of $\rho_6$ with the eigenvalue $\zeta$
and
$$(w, \bar{w}) = -7.$$
Combining $\rho_0$ and $\rho_6$, we define an isometry
$\rho$ of $T$ of order 7 and without nonzero fixed
vectors.
Moreover the action of $\rho$ on the discriminant group $T^*/T$
is trivial.
\end{rem-def}

In Definition-Remark \ref{isometry} we explicitely described an order $7$ isometry $\rho$ on $T$ without nonzero fixed vectors and acting trivially on the discriminant group.
  Hence $\rho$ can be extended to an isometry $\rho$ (we use the same
symbol) of $L_{K3}$ acting trivially on $S$.

Now we consider a $K3$ surface $Y$ with $S_Y \cong S$.  Then the transcendental lattice $T_Y$ of $Y$ is isomorphic to $T$.
We identify $L_{K3}$ and $H^{2}(Y, {\bbZ})$ so that $S = S_Y$ and $T=T_Y$.
If the period $\omega_Y \in T\otimes \bbC$ is an eigenvector of $\rho$, then it follows from the Torelli type theorem for $K3$
surfaces that $\rho$ can be realized as an automorphism $g$ of $Y$ of order 7:  $g^*=\rho$.
Now consider the eigenspace decomposition of $\rho$:
$$T\otimes \bbC = \oplus_{i=1}^6 V_{\zeta^i}$$
where $V_{\zeta^i}$ is the eigenspace corresponding to the eigenvalue $\zeta^i$.
The period domain of the pair $(Y, g)$ is given by
$$D' = \{ \omega \in \bbP(V_{\zeta}) : \langle \omega, {\bar \omega} \rangle > 0\}.$$
Then the above calculations show that the hermitian form on $V_{\zeta}$ defined by $\langle \omega, \bar{\omega}\rangle$ is given by
$7H(\xi)$.
We define an arithmetic subgroup $\tilde{\Gamma}$ by
$$\tilde{\Gamma} = \{ \varphi \in \operatorname{O}(T) : \varphi \circ \rho = \rho \circ \varphi \}$$
and a subgroup $\tilde{\Gamma'}$ by
$$\tilde{\Gamma' }= \tilde{\Gamma} \cap \Ker \{ \operatorname{O}(T) \to \operatorname{O}(q_T) \}.$$
 Let $\Delta = \bigcup \delta^{\perp}\cap D'$ where $\delta$ moves over all $(-2)$-vectors in $T$.  
Then $(D'\setminus \Delta) /\tilde{\Gamma'}$ is the moduli of the pair $(Y,g)$.
Note that $\rho$ has discriminant $-1$ and is contained in $\tilde{\Gamma'}$.  Moreover $\rho$ acts trivially on $B$.
By using the same method as in \cite{K1}, we have

\subsection{Theorem}
$X\cong D/\Gamma'\cong D'/\tilde{\Gamma'}$.

\subsection{Remark}
Naruki \cite{N} showed that $X$ has an elliptic fibration with three singular fibers of type $I_7$ in the sense of Kodaira and with
7 sections.  Thus $X$ contains 28 smooth rational curves.  In particular $X$ has the Picard number 20.
We can see that $\Delta/\Gamma'$ consists of $28$ curves corresponding to 28 smooth rational curves on $X$.
We omit the proof of this fact here.

\subsection{Remark}\label{re}
Let $Z$ be a $K3$ surface with Picard lattice 
 $$S_Z\cong  \begin{pmatrix}2&1\\1&-2\end{pmatrix}\oplus A_4\oplus A_4$$ 
and with a
non-symplectic automorphism $\sigma$ of order 5 acting trivially on $S_Z$.  The author \cite{K2} showed that the moduli space of ordered
5-points on $\bbP^1$ is isomorphic to the moduli space of the pairs
of such $(Z, \sigma)$.  Moreover these moduli spaces can be written birationally as an arithmetic quotient of a 2-dimensional complex ball  which is isomorphic to the quintic del Pezzo surface.

\vskip 30pt

\noindent
{\small Shigeyuki Kond{$\bar{\rm o}$, Graduate School of Mathematics, Nagoya University, Nagoya,
464-8602, Japan. e-mail {\tt kondo@math.nagoya-u.ac.jp}}

\end{document}